\documentclass[hidelinks,onefignum,onetabnum]{siamart250211}

\usepackage{listings}
\usepackage{adjustbox}
\usepackage[numbers]{natbib}

\usepackage{algorithm}
\usepackage{algpseudocodex} % https://ctan.math.washington.edu/tex-archive/macros/latex/contrib/algpseudocodex/algpseudocodex.pdf
\tikzset{algpxIndentLine/.style={color=gray,dashed}}
\algrenewcommand\algorithmicrequire{\textbf{Initialize}}
\algrenewcommand\algorithmicloop{\textbf{parallel}}

\usepackage{mathtools} % Automatically loads amsmath
\usepackage{amsmath, amssymb, amscd, amsfonts} % Other AMS packages
\usepackage{thmtools}
\usepackage{thm-restate} % So that theorems can be restated in appendix with proofs
\usepackage{cleveref}
\usepackage{comment}

\usepackage{enumitem,amssymb}
\newlist{todolist}{itemize}{2}
\setlist[todolist]{label=$\square$} % Load libraries
\newtheorem{assumption}{Assumption}

\DeclareMathOperator*{\argmin}{arg\,min\,}

\definecolor{darkred}{RGB}{120,0,0}
\definecolor{dg}{RGB}{0,120,0}
\definecolor{darkblue}{RGB}{0,0,120}
\definecolor{darkgreen}{RGB}{0,120,0}

\headers{Gradient Minimization Using High-Order Methods}{Yao Ji and Guanghui Lan}

\title{High-order Accumulative Regularization for Gradient Minimization in Convex Programming
\thanks{This work is partially supported by Air Force Office of Scientific Research grant FA9550-22-1-0447 and American Heart Association grant 23CSA1052735.
}
}

\author{Yao Ji\thanks{H. Milton Stewart School of Industrial and Systems Engineering, Georgia Institute of Technology, Atlanta, GA 30332 USA (\email{yaoji@gatech.edu}).}
\and Guanghui Lan\thanks{H. Milton Stewart School of Industrial and Systems Engineering, Georgia Institute of Technology, Atlanta, GA 30332 USA (\email{george.lan@isye.gatech.edu}).}}

\ifpdf
\hypersetup{
  pdftitle={HIGH-ORDER ACCUMULATIVE REGULARIZATION GRADIENT MINIMIZATION FOR CONVEX PROGRAMMING},
  pdfauthor={Yao Ji and Guanghui Lan}
}
\fi

\begin{document}
\maketitle
\begin{abstract}
This paper develops a unified high-order accumulative regularization (AR) framework for convex and uniformly convex gradient norm minimization.
Existing high-order methods often exhibit a gap: the function-value residual decreases fast, while the gradient norm converges much slower. To close this gap, we introduce AR that systematically transforms the fast function-value residual convergence rate into a fast (matching) gradient norm convergence rate.

Specifically, 
for composite convex problems, to compute  an approximate
solution such that the norm of its (sub)gradient does not exceed $\varepsilon,$
the proposed AR methods match the best corresponding convergence rate for the function-value residual. 
We further extend the framework to uniformly convex settings, establishing linear, superlinear, and sublinear convergence of the gradient norm under different lower curvature conditions.
Moreover, we design parameter-free algorithms that require no input of problem parameters, e.g., the Lipschitz constant of the $p$-th-order gradient, the initial optimality gap and the uniform convexity parameter, and allow an inexact solution for each high-order step. To the best of our knowledge, no parameter-free methods can attain such a fast gradient norm convergence rate which  matches that of the function-value residual in the convex case, and no such parameter-free methods for uniformly convex problems exist.
These results substantially generalize existing parameter-free and inexact high-order methods and recover first-order algorithms as special cases, providing a unified approach for fast gradient minimization across a broad range of smoothness and curvature regimes.

\end{abstract}

\section{Introduction}
In this paper, we first consider the following unconstrained convex optimization problem: 
\begin{equation}\label{eqn:main}
    \min_{x \in \mathbb{R}^n} f(x),
\end{equation}
where $f$ is proper and closed.  
We assume  
that $f(\cdot)$ is $p$-times differentiable with a Lipschitz continuous $p$-th derivative:
\begin{align}\label{eqn:smoothness-p}
    \|D^p f(x) - D^p f(y)\|   \leq L_{p+1}(f) \|x - y\|,\quad\forall x, y\in \mathbb{R}^n,
\end{align}
where $ \|D^p f(x)-D^p f(y)\|=\max\limits_{h\in \mathbb{R}^n}\{|D^p f(x)[h]^p-D^p f(y)[h]^p|:\,\|h\|\leq1\}.$ 
We then extend the analysis to composite and  uniformly convex problems.

The last decade has witnessed an increasing interest in the design and analysis of high-order methods for convex programming.
High-order methods exploit high-order derivative information to achieve faster convergence compared to first-order methods. 
Nesterov and Polyak
established the first global worst-case complexity analysis for the second-order schemes by designing a cubic regularization of Newton method (CNM) \cite{nesterov2006cubic}. They showed that CNM takes at most $\mathcal{O}(\varepsilon^{-1/2})$ iterations to reduce the functional  
residual below a given precision $\varepsilon,$ i.e., $ f(x)-f^*\leq \varepsilon,$ where $f$ is a twice-differentiable  
convex function with a Lipschitz continuous Hessian. Later, Nesterov obtained an improved complexity bound of $\mathcal{O}(\varepsilon^{-1/3})$ by developing an accelerated version of CNM \cite{nesterov2008accelerating}.
After that, Monteiro and Svaiter \cite{monteiro2013accelerated} derived an improved iteration complexity of $\mathcal{O}(\varepsilon^{-2/7}\log \varepsilon^{-1})$ through an accelerated Newton proximal extragradient (ANPE).
Later, a lower bound of $\Omega(\varepsilon^{-2/7})$ for second-order methods was established \cite{arjevani2019oracle}, which implies that ANPE \cite{monteiro2013accelerated} is optimal up to a logarithmic factor.
When it comes to $p>2$, the $p$-th-order tensor method and its accelerated version are developed for convex functions with Lipschitz continuous $p$-th derivatives \eqref{eqn:smoothness-p} in \cite{baes2009estimate, nesterov2021implementable}.
Specifically, to achieve $f(x)-f^*\leq \varepsilon,$ the $p$-th-order tensor method takes up to $\mathcal{O}(\varepsilon^{-1/p})$ iterations, while its accelerated version only requires $\mathcal{O}(\varepsilon^{-1/(p+1)})$ iterations \cite{nesterov2021implementable}.
Subsequent methods improved the iteration complexity to $\mathcal{O}(\varepsilon^{-2/(3p+1)}\log \varepsilon^{-1})$, such as those built upon the ANPE method \cite{jiang2021optimal, bubeck2019near, gasnikov2019optimal}, and the inexact high-order proximal-point methods \cite{nesterov2021inexact, nesterov2023inexact}. Recent methods \cite{carmon2022optimal, kovalev2022first} managed to remove the extra logarithmic factor and
thus are optimal.

On the other hand, the worst-case function for the class of functions with Lipschitz continuous Hessian has a discontinuous third derivative \cite{nesterov2023inexact}. 
Therefore, quite interestingly, some algorithms that rely solely on second-order information can surpass the classical lower bound of $\Omega(\varepsilon^{-2/7})$ for the smaller function class whose second and third derivatives are Lipschitz continuous.
For example, a second-order method only requires an iteration complexity of $\mathcal{O}(\varepsilon^{-1/4});$  at each iteration, it calls the second-order oracle once and the first-order
oracle $\mathcal{O}(\log{\varepsilon^{-1}})$ times at most,
see Nesterov \cite{nesterov2021superfast}.
Furthermore, in \cite{nesterov2023inexact}, Nesterov shows that a second-order implementation of the third-order accelerated proximal-point method with segment search can achieve an iteration complexity of $\mathcal{O}(\varepsilon^{-1/5}),$ at each iteration, it calls the second-order oracle once and the first-order
oracle $\mathcal{O}(\log{\varepsilon^{-1}})$ times at most.
Following this, second-order methods have achieved the convergence rate of third-order accelerated tensor methods in terms of function-value residual for composite problems \cite{ahookhosh2024high, ahookhosh2021high}.

Furthermore, the aforementioned methods  assume access to the Lipschitz constant of the corresponding derivative, which is difficult to estimate in practice, and most  assume exact solutions for each subproblem step, which is hard to obtain when $p>3$ \cite{nesterov2021implementable}.
Therefore, it is important to develop parameter-free and inexact algorithms that automatically estimate such parameters and solve each step inexactly, while maintaining similar theoretical guarantees.
Among the second-order methods, Cartis et al. proposed an adaptive cubic regularized Newton method \cite{cartis2012evaluation}, and further relaxed the criterion for solving each subproblem while maintaining the convergence properties, which is widely used due to its numerical efficiency. Recent works \cite{grapiglia2020tensor, grapiglia2022tensor, jiang2020unified, grapiglia2023adaptive,He2025HAR} develop accelerated high-order methods that are fully independent of problem constants while maintaining superior theoretical iteration bounds; they achieve an iteration complexity of 
$\mathcal{O}(\max\{L_0, pL_{p+1}(f), \theta\}{\varepsilon^{-{\frac{1}{{p+1}}}}}),$
where $L_0$ is the initial line-search value and $\theta$ represents the inexactness level of each subproblem.

Despite these advances of high-order methods for convex optimization, the convergence guarantees in existing works  
are mainly established for the function-value residual, i.e., generating a point $\widehat{x}$ such that $f(\widehat{x}) - f^\ast \leq \varepsilon$, where $f$ is the objective function, $f^\ast$ is its optimal value, and $\varepsilon > 0$ is a given precision;  but 
not on the gradient norm criterion, i.e., $\|\nabla f(\widehat{x}) \| \leq \varepsilon.$ 
While the former is theoretically appealing, it is difficult to verify in practice since $f^*$ is unknown. In contrast, the gradient norm criterion is easily computable and may serve as
a measure of optimality. Moreover, the gradient norm minimization perspective also provides a practical mechanism to monitor the progress of the algorithm and  can be used to estimate problem parameters adaptively.
Some works provide complexity results for the gradient norm criterion; however, they do not match the guarantees for the function-value residual.  For example,
in \cite{grapiglia2020tensor}, the authors studied a regularized accelerated tensor method and established a complexity of $\mathcal{O}(\varepsilon^{-1/(p+1)} \log \varepsilon^{-1})$,  
which is worse than the corresponding function-value residual complexity of the accelerated tensor methods \cite{nesterov2021implementable} by a logarithmic factor.
As another example, the ANPE method \cite{monteiro2013accelerated} exhibits a gradient norm complexity of $\mathcal{O}(\varepsilon^{-1/3}\log\varepsilon^{-1})$, which is worse than its nearly optimal function-value residual complexity of $\mathcal{O}(\varepsilon^{-2/7}\log\varepsilon^{-1}).$
A third example is a parameter-free and inexact method, where the best-known iteration complexity for gradient norm minimization is $\mathcal{O}(\max\{L_0, pL_{p+1}(f), \theta\}{\varepsilon^{-{\frac{p+1}{{p(p+2)}}}}})$ \cite{grapiglia2022tensor,grapiglia2023adaptive}. However, this complexity is worse than the function-value residual complexity, which is $\mathcal{O}(\max\{L_0, pL_{p+1}(f), \theta\}\varepsilon^{-\frac{1}{p+1}}).$

To summarize, the convergence rates in terms of gradient norm for different types of high-order methods do not match the corresponding function-value residual. 
This raises a natural question: is there a unified approach to translate the {\it fast} function-value residual rate into a {\it matching} rate for gradient norm minimization?

To address this question, we propose a unified {\it accumulative regularization} (AR) method for gradient minimization and use it to accelerate existing high-order algorithms, yielding best-known gradient norm convergence rates for convex problems. Specifically, for an existing high-order algorithm $\mathcal{A}$, we utilize its {\it fast} function-value residual decay together with its {\it slow} gradient norm decay, and design new gradient norm minimization algorithms that match the convergence rate of the function-value residual.
It is worth noting that 
our approach is inspired by the first-order gradient minimization algorithm of Lan et al. \cite{lan2024optimalparameterfreegradientminimization}, and recovers their results when $p=1$. Our contributions can be briefly summarized as follows.

First, for a twice-differentiable convex function with a Lipschitz continuous Hessian, we refine the estimating sequence analysis of the accelerated CNM (ACNM) to obtain a slow gradient norm convergence rate of $\mathcal{O}(\varepsilon^{-1/2})$, together with its original fast function-value residual decrease $\mathcal{O}(\varepsilon^{-1/3})$. Then we design a third-order AR method that uses ACNM as a subroutine. 
The method restarts intermittently, with each epoch initialized with the output from the previous epoch. Moreover, it applies a single accumulative regularization scheme whose parameter depends on $\varepsilon$ at initialization and then increases geometrically across epochs.
We show that, to achieve $\|\nabla f(\widehat{x})\|\le \varepsilon$, the number of iterations reduces to $\mathcal{O}(\varepsilon^{-1/3})$, thereby matching the complexity for the function-value residual.

Second, we generalize the method to solve the composite convex problem 
\begin{equation}\label{eqn:composite-problem}
  \min\limits_{x \in \mathbb{R}^n} \left\{ f(x) := g(x) + h(x) \right\},
\end{equation}
where $g$ and $h$ are proper closed convex functions, $h$ may be  nonsmooth,
and $g$ is $p$-times continuously  
differentiable on $\mathbb{R}^n.$ We assume that there exists at least one optimal solution $x^*.$  
Furthermore, the level of smoothness of  
$g$ is characterized by the family of Hölder constants, i.e., $\exists\, p \geq 1,\ \nu \in [0\,\,1],\ L_{p,\nu}(g) > 0$ such that:
\begin{equation}\label{Holder}
    \|D^p g(x) - D^p g(y)\|   \leq L_{p,\nu}(g) \|x - y\|^{\nu},\quad \forall x, y \in \mathbb{R}^n.
\end{equation}
The goal is to generate an approximate solution  $\widehat{x}$ such that $\|\widehat{\nu}\|\leq \varepsilon,$ where $\widehat{\nu}\in \partial f(\widehat{x}).$
We design a $(p+\nu)$-th order AR framework that uses different subroutines $\mathcal{A}.$
Within this framework, we accelerate the subgradient norm convergence to match the fast function-value residual rate of $\mathcal{A}.$ We illustrate it with several examples.
\smallskip
\begin{enumerate}
\item[(a)] Suppose $g$ satisfies \eqref{Holder} with general $p\geq 1, \nu\in [0\,\,1],$ and $p+\nu\geq 2,$ by choosing $\mathcal{A}$ as  the $p$-th-order accelerated tensor method \cite{grapiglia2023adaptive}, the $(p+\nu)$-th order AR framework only requires $\mathcal{O}(\varepsilon^{-1/{(p+\nu)}})$ 
iterations, which matches that of the function-value residual \cite{nesterov2021implementable}. 
\item[(b)]
Suppose $g$ satisfies \eqref{Holder} with general $p\geq1,$ and
$\nu = 1$, by choosing $\mathcal{A}$ as the nearly optimal tensor methods
\cite{jiang2021optimal, bubeck2019near, gasnikov2019optimal} or
$p$-th-order inexact proximal point method \cite{nesterov2023inexact,ahookhosh2021high}, the $(p+1)$-th order AR framework requires 
$\mathcal{O}\left(\varepsilon^{-\frac{2}{3p+1}} \log \tfrac{1}{\varepsilon}\right)$  
iterations, which matches that of the function-value residual. 
The extra 
logarithmic term disappears if $\mathcal{A}$ is chosen as
the optimal $p$-th-order tensor method (for function-residual) \cite{carmon2022optimal,kovalev2022first} and can output a slow-gradient norm.

\item[(c)]
Suppose $g$ satisfies \eqref{Holder} with $p=3, \nu=1,$ by choosing $\mathcal{A}$ as the second-order method \cite{nesterov2023inexact}, 
the fourth-order AR framework achieves an iteration complexity of $\mathcal{O}(\varepsilon^{-{1}/{5}});$ at each iteration, it calls the second-order once and the first-order
oracle $\mathcal{O}(\log{\varepsilon^{-1}})$ times at most.
\end{enumerate}

Third, we develop an inexact and parameter-free high-order AR framework 
for unconstrained convex optimization \eqref{eqn:main} where  $f$ satisfies \eqref{Holder} with general $p\geq 1, \nu\in [0\,\,1],$ and $p+\nu\geq 2.$ By choosing $\mathcal{A}$ as adaptive and inexact tensor methods \cite{grapiglia2020tensor, grapiglia2023adaptive}, the $(p+\nu)$-th order inexact and parameter-free AR framework only requires
\begin{align*}
\mathcal{O}\left(\left[\tfrac{\max\{{pL_{p,\nu}}(f), L_0, \theta \}[C_p\operatorname{dist}(x_0, X^*)^{p+\nu-1}]}{\varepsilon}\right]^{\frac{1}{p+\nu}}\right)
\end{align*}
calls to the $p$-th-order oracle, 
where $L_0$ is the initial line-search value and $\theta$ represents the inexactness level of the subproblem of each iteration. It does not
require knowledge of the number of iterations to run in advance, the number of restarting epochs, or $L_{p,\nu} (f),$ and the initial optimality gap $\operatorname{dist}(x_0, X^*)$. It 
matches the best-known convergence rate in terms of function-value residual for adaptive and inexact methods \cite{jiang2020unified, grapiglia2020tensor, grapiglia2023adaptive,He2025HAR}.

Fourth, we consider the composite convex problem \eqref{eqn:composite-problem}, additionally assuming that
$f$ is uniformly convex of order $q$ with parameter $\sigma_q:$
\begin{align*}
    f(y)\geq f(x)+\langle\nabla f(x), y-x\rangle+\tfrac{\sigma_q}{q}\|y-x\|^q,\quad \forall\,\,\, x, y\in\mathbb{R}^n.
\end{align*}
We show that by properly restarting the high-order AR framework, we can obtain an approximate solution $\widehat{x}$ such that $\|\widehat{\nu}\|\leq \varepsilon,$ \text{where} $\widehat{\nu}\in \partial f(\widehat{x})$ 
within at most  
\begin{align*}
\mathcal{O}\left(\left(\tfrac{(p+1)L_{p+1}(g)}{\sigma_{p+1}(f)}\right)^{\frac{2}{3p+1}}  \log \tfrac{\|\nu_0\|}{\varepsilon} \right)
\end{align*}  
iterations if $q=p+1,$
where $\nu_0\in\partial f(x_0).$
Notice that when $p = 1$, this recovers the gradient minimization results for strongly convex and smooth function \cite{lan2024optimalparameterfreegradientminimization}.
Furthermore, it improves the condition number dependence established in \cite{doikov2022local,cartis2022evaluation} for linear decay of the function-value residual.
The linear rate improves to superlinear rate if $q<p+1,$ i.e., $\|\widehat{\nu}\|\leq \varepsilon$  
within at most  
\begin{align*}
   \mathcal{O}\left( 
   % \tfrac{1}{\log{\left(\frac{p}{q-1}\right)}}
   {\log\log\left[\tfrac{\sigma_q(f)}{{q\varepsilon}}{\left(\tfrac{\sigma_q(f)}{qL_{p+1}(g)}\right)^{\frac{q-1}{p+1-q}}}\right]}+\left(\tfrac{qL_{p+1}(g)}{{\sigma_{q}(f)}}\right)^{\frac{2}{3p+1}}\left(\tfrac{q\|\nu_0\|}{\sigma_{q}(f)}\right)^{{\frac{2(p-q+1)}{(3p+1)(q-1)}}}\right)
\end{align*}
iterations. When $p=q=2,$ the algorithm achieves the iteration complexity of $\mathcal{O}( {\log\log{\tfrac{\sigma^2_2(f)}{L_{3}(g)\varepsilon}}}+[{L_{3}(g)\|\nu_0\|}/{\sigma^2_{2}(f)}]^{{\frac{2}{7}}}),$ which nearly matches the lower bound \cite{arjevani2019oracle} for the function-value residual.
The algorithm becomes sublinear if $q>p+1,$ and requires
\begin{align*}
  \mathcal{O}\left(  \left(\tfrac{qL_{p+1}(g)}{{\sigma_{q}(f)}}\right)^{\frac{2}{3p+1}}\left(\tfrac{q}{\sigma_{q}(f)}\right)^{{\frac{2(q-1-p)}{(3p+1)(q-1)}}}{\varepsilon}^{-{\frac{2(q-1-p)}{(3p+1)(q-1)}}}\right)
\end{align*}
iterations. It improves over the 
sublinear rate $\mathcal{O}(\varepsilon^{-\frac{q-p-1}{pq}})$ when $q>p+1$ for the function-value residual \cite{cartis2022evaluation}. The results in \cite{doikov2022local} are  non-asymptotic rates, but assume knowledge of the Lipschitz constant
$L_{p+1}(q)$ and uniform convexity $\sigma_{q}(f)$; a very recent work \cite{Welzel2025Local} removes the need to know the Lipschitz constant $L_{p+1}(q)$ through an adaptive algorithm, but it is an asymptotic rate.

 Lastly, we develop inexact and parameter-free gradient minimization which  {\it does not} require knowledge of the Lipschitz continuous parameter or the uniform convex parameter.
 Using the unconstrained uniformly convex optimization problem with $q=p+1$ as a showcase, 
our parameter-free and inexact method requires
    \begin{align*}
\mathcal{O}\left(\left(\tfrac{\max\{{pL_{p+1}}(f), L_0, \theta \}}{\sigma_{p+1}(f)}\right)^{\frac{1}{p+1}}\left\{\left\lceil \log_4\tfrac{\sigma_{p+1,0}}{\sigma_{p+1}(f)}\right\rceil+\left\lceil \log_2\tfrac{\|\nabla f(x_0)\|}{\varepsilon}\right\rceil\right\}\right)
    \end{align*}
    iterations to compute a solution $\widehat{x}$ such that $\|\nabla f(\widehat{x})\|\leq \varepsilon,$ where $\sigma_{p+1,0}$ is the initial guess of the uniform convexity parameter $\sigma_{p+1}(f).$
    It is the first parameter-free and inexact method that does not require the knowledge of the uniform convexity parameter $\sigma_{p+1}(f),$ while still matching the fastest known non-asymptotic convergence rate of methods that assume $L_{p+1}(f), \sigma_{p+1}(f)$ are known.

To the best of our knowledge, all of the above gradient norm convergence results are new.
The only exception is the recent work \cite{dvurechensky2024near}, which established a complexity of $\mathcal{O}(\varepsilon^{-\frac{2}{3p+1}}\log \tfrac{1}{\varepsilon} )$ for convex problems.  However, the result appears to be more restrictive as it requires $p\geq 2,$ thus 
hard to expand to the cases of bounded Hessian $(p=2, \nu=0)$ or $L$-smooth $(p=1, \nu=1).$ Moreover, it is restricted to the unconstrained case and does not extend to general convex composite or uniformly convex objectives. Unlike AR, it relies on known problem parameters and requires the subproblem in each iteration to be solved exactly.

The rest of this paper is organized as follows.  
In \Cref{sec:second-order-methods}, we provide a brief review of the ACNM method and its modified analysis for convex functions with smooth Hessians. Then, we construct a third-order AR approach to accelerate the gradient norm convergence rate.
In \Cref{sec:Optimal high-order methods for structured convex problems}, we propose a general AR framework and employ different types of subroutines $\mathcal{A}$ for structured convex composite problems.  In \Cref{sec:parameter-free-convex},  we derive a parameter-free, inexact, high-order AR framework 
for convex functions. 
In \Cref{sec:Global high-order non-degeneracy}, we further investigate gradient norm minimization for uniformly convex functions and develop a parameter-free and inexact algorithm.

\subsection{Notation and terminology}
We use $\|\cdot\|$ to denote the Euclidean norm in $\mathbb{R}^n,$ which is associated with the inner product $\langle \cdot, \cdot\rangle.$
Denote the  directional derivative of function \( f \) at \( x \) along directions \( h_i \in \mathbb{R}^n, i = 1, \dots, p \) as $D^p f(x)[h_1, \dots, h_p].$
Note that \( D^p f(x)[\cdot] \) is a symmetric \( p \)-linear form.
For example, for any $x\in\text{dom } f$ and $h_1, h_2\in\mathbb{R}^n,$ we have 
\begin{align*}
    D f(x)[h_1]=\langle \nabla f(x),  h_1\rangle\quad\text{and}\quad D^2 f(x)[h_1, h_2]=\langle \nabla^2 f(x)h_1,  h_2\rangle.
 \end{align*}
   Its norm is induced by the Euclidean norm.
\begin{equation*}
    \|D^p f(x)\|   \coloneqq \max_{h_1, \dots, h_p} \left\{
        D^p f(x)[h_1, \dots, h_p]: \|h_i\| \leq 1, i = 1, \dots, p
    \right\}.
\end{equation*}
For a convex function $f,$ $\partial f(x)$ denotes the subdifferential at $x$. For any real number $s,$ $\lceil s\rceil$ and $\lfloor s\rfloor$ denote the nearest integers to $s$ from above and below, respectively. Let $[m]\triangleq\{1,\dots,m\}$, with $m\in \mathbb{N}_{+}.$

\section{Acceleration of the Cubic Regularization of Newton’s Method}\label{sec:second-order-methods}
In this section, we consider the unconstrained convex optimization problem \eqref{eqn:main},
where $f$ is a twice-differentiable convex function on $\mathbb{R}^n$ 
that satisfies \eqref{eqn:smoothness-p} with $p=2,$ i.e., there exists $L_3(f)>0$ such that 
$ \| \nabla^2 f(x)-\nabla^2 f(y)\|\leq L_3(f)\|x-y\|,$ for all $x, y\in\mathbb{R}^n.$

We start with a modified version of the ACNM \cite{nesterov2008accelerating}, and show that unlike the ACNM, which only outputs the function value,
the modified ACNM outputs both a gradient norm bound (with a slower rate) and the original fast function-value residual, simultaneously. Define
$T_M(x) \coloneqq \arg\min_{y \in \mathbb{R}^n} \xi_{2, x}(y)$, where  
\begin{align*}
\xi_{2, x}(y)={f(x)+\langle\nabla f(x), y-x\rangle+\tfrac{1}{2}\langle \nabla^2 f(x) (y-x), y-x\rangle}+\tfrac{M}{6}\|y-x\|^3,
 \end{align*}
 and $\arg\min$ indicates that $T_M(x)$ is chosen from the set of global minimizers of the corresponding problem.
 \begin{algorithm}[!ht]
\caption{Accelerated Cubic Regularization of Newton's Method}\label{alg:inner-loop}
\begin{algorithmic}[1]
 \Require $x_0\in\mathbb{R}^n, \{a_k\}_{k\geq 1}, $ $a_k>0, A_1>0.$
 Compute $x_1=T_{L_3(f)}(x_0)$ and define
\begin{equation*}
f_1(x)\coloneqq {f(x_1)+{\tfrac{1}{\sqrt{L_3(f)+M}}\left\|\nabla f(x_1)\right\| ^{\frac{3}{2}}}}+\tfrac{C}{6}\|x-x_0\|^3.
 \end{equation*}
  \For{$k\geq 1$} 
 \State Compute $\nu_k=\argmin_{x\in\mathbb{R}^n}f_k(x).$ Choose  $A_{k+1}=A_k+a_k$ and 
 \begin{align*}     y_k&=(1-\alpha_k)x_k+\alpha_k\nu_k,\quad \text{where}\quad\alpha_k=\tfrac{a_k}{A_{k+1}},\quad a_k>0.
 \end{align*}
 \State Compute $ x_{k+1}=T_{M}(y_k)$ and update
\begin{align*}
     f_{k+1}(x)&=f_k(x)+a_k\left( f(x_{k+1})+\langle \nabla f(x_{k+1}), x-x_{k+1}\rangle\right).
 \end{align*}
 \EndFor
\end{algorithmic}
\end{algorithm}
With a modified ACNM analysis, we can establish the following relations:
\begin{align*}
\mathcal{R}_k^1: \quad   & A_k f(x_k)+\textstyle\sum_{j=1}^k{A_{j}\tfrac{\left\|\nabla f(x_{j})\right\| ^{\frac{3}{2}}}{\sqrt{L_3(f)+M}}}\leq f_k^*\equiv \min\limits_{x\in \mathbb{R}^n}f_k(x).\\
    \mathcal{R}_k^2: \quad   & f_k(x)\leq A_k f(x)+\tfrac{2L_3(f)+C}{6}\|x-x_0\|^3+\tfrac{\left\|\nabla f(x_1)\right\| ^{\frac{3}{2}}}{\sqrt{L_3(f)+M}}.
\end{align*}
Here, $\mathcal{R}_k^1$ shows a lower bound for the minimum function value of the estimating function $f_k(x),$  and $\mathcal{R}_k^2$ establishes an upper bound for $f_k(x).$  See \Cref{sec:Gradient Complexity of the Subroutine} for the proof. Thus, we have the following convergence guarantee for the modified ACNM.
\begin{lemma}\label{lem:prop-cubic}
    Let the sequence $\{x_k\}^{\infty}_{k=1}$ be generated by  \Cref{alg:inner-loop} with the parameters $ M=2L_3(f), C={12L_3(f)}/{(\sqrt{2}-1)^2}, a_k={(k+1)(k+2)}/{2}, A_1=1,$
    then, for any $k\geq 1,$ we have
    \begin{align*}
         &f(x_k)-f(x^*)+\tfrac{1}{\sqrt{3L_3(f)}}\|\nabla f(x_k)\| ^{\frac{3}{2}}
        \leq\tfrac{80L_3(f)\|x_0-x^*\|^3}{k(k+1)(k+2)},
    \end{align*}
    where $x^*$ is an optimal solution to the problem \eqref{eqn:main}.
\end{lemma}
    Note that the convergence rate of the gradient norm is slower than that of the function-value residual.  
To achieve a point $\tilde{x}$ such that $\|\nabla f(\tilde{x})\|\leq \varepsilon,$ \Cref{alg:inner-loop} requires $\mathcal{O}(L_3(f)^{1/2} \|x_0 - x^*\| / \varepsilon^{1/2})$ iterations, while to achieve the function-value residual 
$f(\tilde{x})-f^*\leq \varepsilon,$ \Cref{alg:inner-loop} requires $\mathcal{O}(L_3(f)^{1/3} \|x_0 - x^*\| / \varepsilon^{1/3})$ iterations. Note that using a regularization technique could improve the convergence rate to $\mathcal{O}( [L_3(f)^{1/3}\|x_0 - x^*\|^{2/3} / \varepsilon^{1/3}]\log({L_3(f)\|x_0 - x^*\|^2}/{\varepsilon}))$ \cite{nesterov2012make}. However, the rates differ by a logarithmic term.

We next outline a third-order AR method that utilizes \Cref{alg:inner-loop} to improve its gradient norm convergence rate to that of the function-value residual. Notably, the design of \Cref{alg:main'} is built upon accumulative regularization, where the search points $\{x_s\}_{s \in [S]}$ are generated by inexactly solving a third-order proximal mapping defined in \eqref{eqn:sub-proximal-p'}. \Cref{alg:main'} requires a sequence of regularization parameters $\{\sigma_s\}_{s \in [S]}$, and the previous search points $\{x_s\}_{s \in [S-1]}$ are accumulated within the regularization term. 
Unlike the fixed regularization \cite{nesterov2012make}, \Cref{alg:main'} uses an accumulative, restart-based scheme: each epoch starts at $x_{s-1}$, and the regularizer adds a new term centered at $x_s$ while preserving all earlier terms (rather than centered at $x_0$). Such accumulative regularization is close in spirit to the classic accelerated proximal point method \cite{guler1992new}.

Observe that the proximal term is cubic. We next develop several properties of this cubic term for use in the subsequent algorithmic analysis.
Denote the cubic function as $d_3(x) = \|x - x_0\|^3/3$, for any $x_0 \in \mathbb{R}^n.$ Note that it has a Lipschitz continuous Hessian \cite{rodomanov2020smoothness} as follows:
\begin{equation}\label{eqn:power prox-function-smooth-3}
     \|D^2 d_{3}(x)-D^2 d_{3}(y)\|  \leq 4\|x-y\|,\quad \forall \,\, x, y\in\mathbb{R}^n.
\end{equation}
Furthermore, it is uniformly convex 
of degree $3$ with parameter $1/2$ \cite{nesterov2018lectures}. 
\begin{equation}\label{eqn:power prox-function-3}
d_{3}(x)-d_{3}(y)-\langle \nabla d_{3}(y), x-y\rangle\geq \tfrac{1}{6}\|x-y\|^{3},\quad \forall \,\, x, y\in\mathbb{R}^n.
\end{equation}
\begin{algorithm}[!ht]
\caption{A third-order AR method for gradient minimization}\label{alg:main'}
\begin{algorithmic}[1]
 \Require Total number of epochs $S;$ strictly increasing regularization parameters $\{\sigma_s\}_{s=0}^S$ with $\sigma_0=0;$ initial point $x_0\in\mathbb{R}^n.$
  \For{$s=1,\dots,S$} 
 \State Set initialization at epoch $s$ to $x_{s-1}.$
    \State 
    Compute an approximate solution $x_s$   of the proximal subproblem
    by running \Cref{alg:inner-loop}  with the initialization $x_{s-1}$ for $N_s$ iterations.
    \begin{align}\label{eqn:sub-proximal-p'}
        x_s\approx   & \argmin\limits_{x\in\mathbb{R}^n} \left\{ f_s(x)\coloneqq f(x)+\textstyle\sum_{i=1}^s\tfrac{(\sigma_i-\sigma_{i-1})\|x-{x}_{i-1}\|^3}{3}\right\}.
    \end{align}
     \EndFor
 \State \Return{$x_S$}
\end{algorithmic}
\end{algorithm}

We now present the convergence rate of the third-order AR method \Cref{alg:main'}.
\begin{proposition}\label{ACN}
Suppose the parameters are set to      \begin{align}\label{eqn:cubic-samplesize'}
S=\left\lceil\log_{4}\tfrac{L_{3}(f) D^2}{\varepsilon}\right\rceil+1,\,\, \sigma_s=\tfrac{4^{s-2}\varepsilon}{D^2},\,\,  N_s=\left\lceil 4\left(\tfrac{480[L_3(f)+4\sigma_s]}{{\sigma_{s}}}\right)^{\frac{1}{3}}\right\rceil,
 \end{align}
  where $D$ is an upper bound on the
distance to the set of optimal solutions, i.e., $D\geq \min\limits_{x^*\in X^*}\|x_0-x^*\|.$
  Then \Cref{alg:main'} can compute an approximate solution $x_S$ such that $ \|\nabla f(x_S)\|\leq \varepsilon$
after at most $$\left\lceil\tfrac{128L_3(f)^{\frac{1}{3}}D^{\frac{2}{3}}}{{\varepsilon^{\frac{1}{3}}}}+128S\right\rceil$$
 evaluations of the first- and second-order information of $f$.
    \end{proposition}
 \begin{proof}
 By \eqref{eqn:sub-proximal-p'}, 
 we have
 $\|\nabla f(x_S)\|
     \leq \|\nabla f_S(x_S)\|+\sum_{i=1}^S(\sigma_i-\sigma_{i-1})\|x_S-{x}_{i-1}\|^2.$
     We start with bounding $ \|\nabla f_S(x_S)\|.$
By the property of the power prox-function in \eqref{eqn:power prox-function-smooth-3}, we conclude that $f_S(\cdot)$ has a Lipschitz continuous Hessian with parameter $L_{3}(f_S)=L_3(f)+4\sigma_S.$ 
     For all $s\in[S],$
     let $x_s^*$  denote the exact solution.
 By \Cref{lem:prop-cubic}, 
\begin{align}\label{eqn:cubic-gradient-overall}
   \|\nabla f_S(x_S)\|&{\leq}\,{\left[3L_{3}(f_S)\right]^{\frac{1}{3}}\left(\tfrac{80L_{3}(f_S)\|x_S^*-x_{S-1}\|^3 }{N_S^3}\right)^{\frac{2}{3}}}.
 \end{align} 
We proceed with  bounding $\|x_S^*-x_{S-1}\|.$
By optimality conditions of \eqref{eqn:sub-proximal-p'} at $x_{s-1}^*$ and $x_{s}^*,$ the following hold
    \begin{align}\label{eqn:optimality-condition-each epoch}
        &f(x_{s-1}^*)+\textstyle\sum_{i=1}^{s-1}\tfrac{(\sigma_i-\sigma_{i-1})\|x_{s-1}^*-{x}_{i-1}\|^3}{3}+\tfrac{(\sigma_s-\sigma_{s-1})\|x_{s}^*-{x}_{s-1}\|^3}{3}\notag\\
        &\leq f(x_{s}^*)+\textstyle\sum_{i=1}^{s-1}\tfrac{(\sigma_i-\sigma_{i-1})\|x_{s}^*-{x}_{i-1}\|^3}{3}+\tfrac{(\sigma_s-\sigma_{s-1})\|x_{s}^*-{x}_{s-1}\|^3}{3}\notag\\
&\leq f(x_{s-1}^*)+\textstyle\sum_{i=1}^{s-1}\tfrac{(\sigma_i-\sigma_{i-1})\|x_{s-1}^*-{x}_{i-1}\|^3}{3}+\tfrac{(\sigma_s-\sigma_{s-1})\|x_{s-1}^*-{x}_{s-1}\|^3}{3}.
    \end{align}
    Thus, we have $ \|x_s^*-x_{s-1}\|\leq \|x_{s-1}^*-x_{s-1}\|,$ for all $s\in[S].$ 
Using this fact, we can derive the linear convergence of the solution error per epoch, i.e., $\|x_s-x_s^*\|\leq\|x_{s-1}^*-x_{s-1}\|/4$ as follows.
\begin{align}\label{eqn:linear-cubic-per-epoch}
\|x_s-x_s^*\|^3&\overset{\text{(a)}}{\leq}\tfrac{6[f_s(x_s)-f_s(x_s^*)]}{\sigma_s} \overset{\text{(b)}}{\leq}\tfrac{6}{\sigma_s}\tfrac{80L_{3}(f_s) }{N_s^3}\|x_s^*-x_{s-1}\|^3\notag\\
&\overset{\text{(c)}}{\leq}\tfrac{6}{\sigma_s}\tfrac{80L_{3}(f_s) }{N_s^3}\|x_{s-1}^*-x_{s-1}\|^3\overset{\text{(d)}}{\leq}\tfrac{1}{64}\|x_{s-1}^*-x_{s-1}\|^3,
 \end{align}
 where in (a), we used the convexity of $f,$  the
 uniform convexity of the power prox-function in \eqref{eqn:power prox-function-3},
 and $\nabla f_s(x_s^*)=0,$ in (b) we used \Cref{lem:prop-cubic}, in (c), we used the fact $ \|x_s^*-x_{s-1}\|\leq \|x_{s-1}^*-x_{s-1}\|,$ and in (d), we substituted the choice of $N_s$ in \eqref{eqn:cubic-samplesize'}. Thus, substituting  \eqref{eqn:linear-cubic-per-epoch} and \eqref{eqn:cubic-samplesize'}
 into
 \eqref{eqn:cubic-gradient-overall}, we have
\begin{align*}
 \|\nabla f_S(x_S)\|^3
 &\leq{3L_{3}(f_S)\left(\tfrac{80L_{3}(f_S)\|x_{S-1}^*-x_{S-1}\|^3 }{N_S^3}\right)^{2}}\leq
 3(L_{3}(f)+\tfrac{4^{S-1}\varepsilon}{D^2})(\tfrac{\varepsilon D}{96\times4^{2S}})^{2}\leq \tfrac{\varepsilon^3}{8}.
 \end{align*}
 It remains to bound 
 $\textstyle\sum_{i=1}^{S}(\sigma_i-\sigma_{i-1})\|x_S-{x}_{i-1}\|^2.$
\begin{align*}
&\textstyle\sum_{i=1}^{S}(\sigma_i-\sigma_{i-1})\|x_S-{x}_{i-1}\|^2\notag\\
% =\sum\limits_{i=1}^S(\sigma_i-\sigma_{i-1})\|\sum\limits_{k=i}^{s}(x_k-x_{k-1})\|^2\notag\\
&=\textstyle\sum_{i=1}^{S}(\sigma_i-\sigma_{i-1})\bigg(2\textstyle\sum_{\ell>k\geq i}^S\langle x_k-x_{k-1}, x_{\ell}-x_{\ell-1}\rangle+\textstyle\sum_{k=i}^{S}\left\|x_k-x_{k-1}\right\|^2\bigg)\notag\\
&\leq\textstyle\sum_{i=1}^{S}(\sigma_i-\sigma_{i-1})\bigg(2\textstyle\sum_{\ell>k\geq i}^S\|x_k-x_{k-1}\|\| x_{\ell}-x_{\ell-1}\|+\textstyle\sum_{k=i}^{S}\left\|x_k-x_{k-1}\right\|^2\bigg).
 \end{align*}
 By using the fact $\|x_k^*-x_{k-1}\|\leq \|x_{k-1}^*-x_{k-1}\|,$ we have
\begin{align*}
&2\textstyle\sum_{\ell>k\geq i}^S\|x_k-x_{k-1}\|\| x_{\ell}-x_{\ell-1}\|
     % &\overset{\triangle}{\leq}2\sum\limits_{\ell>k\geq i}^s\left(\|x_k-x_k^*\|+\|x_k^*-x_{k-1}\|\right)\left(\| x_{\ell}-x_{\ell}^*\|+\|x_{\ell}^*-x_{\ell-1}\|\right)\notag\\
\leq8\textstyle\sum_{\ell>k\geq i}^S\|x_{k-1}^*-x_{k-1}\|\|x_{\ell-1}^*-x_{\ell-1}\|\notag\\
&\overset{\eqref{eqn:linear-cubic-per-epoch}}{\leq}8\textstyle\sum_{\ell>k\geq i}^S\tfrac{1}{4^{k-1}}\tfrac{D^2}{4^{\ell-1}}\leq8\textstyle\sum_{k=i}^S\tfrac{D^2}{4^{k-1}}\tfrac{1/4^{k}}{1-1/4}\leq\textstyle\sum_{k=i}^S\tfrac{D^2}{4^{2k-2}}\tfrac{8}{3}.
 \end{align*}
 Similarly, we have
$\sum_{k=i}^{S}\left\|x_k-x_{k-1}\right\|^2\leq 2\sum_{k=i}^{S}{D^2}/{4^{2k}}.$
 Therefore,
 we have
\begin{align*}
&\textstyle\sum_{i=1}^S(\sigma_i-\sigma_{i-1})\|x_S-{x}_{i-1}\|^2
    \leq \textstyle\sum_{i=1}^S(\sigma_i-\sigma_{i-1})\left(\textstyle\sum_{k=i}^{S}\tfrac{2D^2}{4^{2k}}+\textstyle\sum_{k=i}^S\tfrac{D^2}{4^{2k-2}}\tfrac{8}{3}\right)\notag\\
    &\leq 50\textstyle\sum_{i=1}^S(\sigma_i-\sigma_{i-1})\textstyle\sum_{k=i}^{S}\tfrac{D^2}{4^{2k}}
   \overset{\eqref{eqn:cubic-samplesize'}}{\leq} \tfrac{\varepsilon}{2}.
 \end{align*}
 Combining this with $\|\nabla f(x_S)\|
     \leq \|\nabla f_S(x_S)\|+\sum_{i=1}^S(\sigma_i-\sigma_{i-1})\|x_S-{x}_{i-1}\|^2,$ we conclude $\|\nabla f(x_S)\|\leq \varepsilon.$
The total number of first- and second-order evaluations of $f$ is bounded by
\begin{align*}
   \textstyle\sum_{s=1}^S N_s&
    % = \sum\limits_{k=1}^S\left\lceil\frac{32[L_3(f)+2\sigma_k]^{\frac{1}{3}}10^{\frac{1}{3}}}{{\sigma_{k-1}}^{\frac{1}{3}}}\right\rceil
    \overset{\textnormal{(e)}}{\leq}\textstyle\sum_{s=1}^S\left\lceil\tfrac{4[L_3(f)^{\frac{1}{3}}+\left({4^{s-1}\varepsilon}/{D^2}\right)^{\frac{1}{3}}]480^{\frac{1}{3}}}{\left({4^{s-2}\varepsilon}/{D^2}\right)^{\frac{1}{3}}}\right\rceil
    % \notag\\
    % &\leq\sum\limits_{k=1}^S\left\lceil\frac{640^{\frac{1}{3}}L_3(f)^{\frac{1}{3}}D^{\frac{2}{3}}}{2^{k-1}\varepsilon^{\frac{1}{3}}}+640^{\frac{1}{3}}\right\rceil
    \leq\left\lceil\tfrac{128L_3(f)^{\frac{1}{3}}D^{\frac{2}{3}}+128S\varepsilon^{\frac{1}{3}}}{\varepsilon^{\frac{1}{3}}}\right\rceil,
\end{align*}
where in (e), we substituted the choices of $\sigma_k, N_k$ in \eqref{eqn:cubic-samplesize'} and used the relation $(a+b)^{1/3}\leq a^{1/3}+b^{1/3},$ for all $a, b\geq0.$
This concludes the proof.
 \end{proof}
The following comments are in order.

\noindent {\bf (i)} \textbf{On the parameter choices:} The regularization parameter $\sigma_s$ is exponentially increasing, while the epoch length $N_s$ is exponentially decreasing, scaling inversely with respect to $\sigma_s$, i.e., $N_s = \Theta([L_3(f)/\sigma_s]^{1/3})$.  
The dominant iterations belong to the first epoch, which determines the final convergence rate, as the remaining $N_s$ are summable.
Furthermore, the number of epochs $S$ is chosen such that $4\sigma_S$ reaches $L_3(f)$, after which the algorithm terminates.

\noindent {\bf (ii)} \textbf{On the convergence rate:}
Observe that the convergence rate in terms of the gradient norm in \Cref{ACN} 
removes the extra $\log({L_3(f)\|x_0 - x^*\|^2}/{\varepsilon})$  in \cite{nesterov2012make}, and matches the convergence rate of the function-value residual for ACNM \cite{nesterov2008accelerating}.  
Furthermore, the regularization parameter $\sigma_s$ increases geometrically as the algorithm approaches the true solution $x^*$,  which enables the algorithm to use shorter epoch lengths $N_s$, 
thus achieving a faster overall convergence rate.

\section{Gradient Minimization Framework for Structured Convex Problems}\label{sec:Optimal high-order methods for structured convex problems}

In this section, we study the convex composite problem \eqref{eqn:composite-problem}.
We first introduce a general high-order AR framework, and then in \Cref{sec:Minimization of convex Hölder smooth functions} we study its convergence when the smooth part has $p$-th derivatives that are $\nu$-Hölder continuous.  
In \Cref{Sec:Optimal gradient minimization}, we focus on the case when the smooth function part has Lipschitz continuous $p$-th-order  derivative (i.e., $\nu=1$).

Consider the high-order AR framework in \Cref{alg:main-composite}.
 \begin{algorithm}[!ht]
\caption{High-order AR framework for gradient minimization}\label{alg:main-composite}
\begin{algorithmic}[1]
 \Require Total number of epochs $S;$ strictly increasing regularization parameters $\{\sigma_s\}_{s=0}^S$ with $\sigma_0=0;$  initial point $x_0\in\mathbb{R}^n.$
\For{$s=1,\dots,S$} 
    \State 
    Compute an approximate solution $x_s$   of the proximal subproblem
\begin{align}\label{eqn:sub-prox-composite}
 x_s\approx \argmin\limits_{x\in \mathbb{R}^n} \left\{ f_s(x)\coloneqq f(x)+\textstyle\sum_{i=1}^s\tfrac{\sigma_i-\sigma_{i-1}}{p+\nu}\|x-{x}_{i-1}\|^{p+\nu}\right\},
    \end{align}
    where $p+\nu\geq 2,$ 
    by running some subroutine $\mathcal{A} $ with the  initialization $x_{s-1}$ for $N_s$ iterations.
     \EndFor
 \State \Return{$x_S$}
\end{algorithmic}
\end{algorithm}
Observe that \Cref{alg:main-composite} shares the acceleration idea of \Cref{alg:main'}, and the power of the regularization term depends on the smoothness  level of $g.$ The subroutine $\mathcal{A}$ needs to satisfy some properties for AR to converge, for example, $\mathcal{A}$ outputs a slow gradient norm decay and a fast function-value residual decay, which we will specify later in the convergence analysis.

Notice that the regularization terms in the subproblem \eqref{eqn:sub-prox-composite} build upon the power prox-function $d_{p+\nu}: \mathbb{R}^n \rightarrow \mathbb{R},$ defined as $d_{p+\nu}(x) \coloneqq 1/{(p+\nu)} \|x - x_0\|^{p+\nu}, x_0\in\mathbb{R}^n.$  
We recall its uniform convexity and the $\nu$-Hölder continuity of the $p$-th-order derivatives  as follows \cite{nesterov2008accelerating, rodomanov2020smoothness}.
Note that we only require $p+\nu\geq 2.$ Therefore, AR covers the first order algorithm for $L$-smooth convex function, i.e., $p=\nu=1,$ and the bounded Hessian case, i.e., $p=2, \nu=0.$
\begin{lemma}\label{lem:power-uniform-smoothness}
For all $\nu\in[0 \,\, 1],  {p+\nu\geq 2}, $
$d_{p+\nu}(x)$ is a uniformly convex function of degree $p+\nu$ with parameter $2^{-(p+\nu-2)}.$ 
\begin{equation}\label{eqn:power prox-function}
d_{p+\nu}(x)-d_{p+\nu}(y)-\langle \nabla d_{p+\nu}(y), x-y\rangle\geq \tfrac{1}{p+\nu}\left(\tfrac{1}{2}\right)^{p+\nu-2}\|x-y\|^{p+\nu}.
\end{equation}
Furthermore, its $p$-th-order derivatives satisfies $\nu$- Hölder continuity.
\begin{equation}\label{eqn:power prox-function-smooth}
     \|D^p d_{p+\nu}(x)-D^p d_{p+\nu}(y)\|  \leq \tfrac{2}{p+\nu}\textstyle\prod_{i=1}^{p}(i+\nu)\|x-y\|^\nu.
\end{equation}
\end{lemma}
Suppose \Cref{alg:main-composite} terminates at epoch $S$ with a point $x_S^k, k\geq N_S.$ Then,
by the construction of $f_S,$ and suppose $p+\nu\geq 2,$ we have $$\nu_S^k\coloneqq v_S^k-\textstyle\sum_{i=1}^S(\sigma_i-\sigma_{i-1})\|{x}_S^k-x_{i-1}\|^{p+\nu-2}({x}_S^k-x_{i-1})\in\partial f({x}_S^k),$$ where $v_S^k\in \partial f_S(x_S^k).$
Therefore, the subgradient norm $\|\nu_S^k\|$ can be decomposed into two parts as follows.
\begin{align}\label{eqn: error-decompose}
\|\nu_S^k\|&\leq \|v_S^k\|+{\textstyle\sum_{i=1}^S(\sigma_i-\sigma_{i-1})\|x_S^k-{x}_{i-1}\|^{p+\nu-1}}.
 \end{align} 
We next show that the second term is controlled by the regularization parameters, provided the inner subroutine satisfies the linear convergence condition
$\|x_s - x_s^*\| \leq \|x_{s-1} - x_{s-1}^*\|/4,
$ which can be satisfied by our subroutines, as we will demonstrate in the next subsections.

Let $D$ denote an upper bound on the distance to the set of optimal solutions, i.e., 
\begin{align}\label{eqn:initial-optimality-gap}
   \operatorname{dist}(x_0,X^*)\coloneqq \min_{x^* \in X^*} \|x_0 - x^*\|\leq D.
\end{align} 
This quantity characterizes the initial optimality gap and will be used frequently in our convergence analysis. We have the following convergence guarantee for \eqref{eqn: error-decompose}.
\begin{lemma}\label{lem:linear}
    Suppose $p + \nu \geq 2,$  the subroutine can compute $x_s,$ for all $s \in [S]$, such that $\|x_s - x_s^*\| \leq \|x_{s-1} - x_{s-1}^*\|/4,$ and suppose in the last epoch $S$, for all $k \geq N_S$, we have $\|x_S^k - x_S^*\| \leq \|x_{S-1} - x_{S-1}^*\|/4.$
Then, for all $k\geq N_S,$
we have  
\begin{align*}
    \textstyle\sum_{i=1}^S (\sigma_i - \sigma_{i-1}) \|x_S^k - x_{i-1}\|^{p+\nu-1}
    \leq \tfrac{2\sigma_S D^{p+\nu-1}}{4^{(p+\nu-1)(S-2)}} + \tfrac{(18D)^{p+\nu-1}}{2}
    \textstyle\sum_{i=1}^S \tfrac{\sigma_i - \sigma_{i-1}}{4^{i(p+\nu-1)}},
\end{align*}
where $D$ is defined in \eqref{eqn:initial-optimality-gap}.
\end{lemma}
\begin{proof}
Similar to the proof of \Cref{ACN}, by using the optimality condition of \eqref{eqn:sub-prox-composite}, for all $s\geq 1,$ $s\in[S],$ we have
\begin{align}
        \|x_{s-1}-x_s^*\|&\leq \|x_{s-1}-x_{s-1}^*\|.\label{eqn:consecutive-sol-dis}
        \end{align}
        Furthermore, for the last epoch $S,$ if $p+\nu\geq 2,$
        we have
        \begin{align}
&\big\|\textstyle\sum_{k=i}^{S-1}(x_k-x_{k-1})\big\|^2
\leq 2\textstyle\sum_{k=i}^{S-1}\left(\left\|x_k-x_k^*\right\|^2+\left\|x_{k-1}^*-x_{k-1}\right\|^2\right)\notag\\
&\quad\quad\quad\quad+\textstyle\sum_{\ell>k\geq i}^{S-1}\left(\|x_k-x_k^*\|+\|x_{k-1}^*-x_{k-1}\|\right)\left(\| x_{\ell}-x_{\ell}^*\|+\|x_{\ell-1}^*-x_{\ell-1}\|\right).\label{eqn:quadratic-relation}
\end{align}
By the convexity of the function $x\rightarrow \|x\|^{p+\nu-1},$ for $p+\nu\geq 2,$  we have
\begin{align*}
&\textstyle\sum_{i=1}^S(\sigma_i-\sigma_{i-1})\|x_S^k-{x}_{i-1}\|^{p+\nu-1}\notag\\
% =\textstyle\sum_{i=1}^S(\sigma_i-\sigma_{i-1})\left\|\textstyle\sum_{k=i}^{S-1}(y_k-y_{k-1})+y_S^k-y_{S-1}\right\|^{p+\nu-1}\notag\\
&\leq2^{p+\nu-2}\left(\sigma_S\left\|x_S^k-x_{S-1}\right\|^{p+\nu-1}+\textstyle\sum_{i=1}^S(\sigma_i-\sigma_{i-1})\big\|\textstyle\sum_{k=i}^{S-1}(x_k-x_{k-1})\big\|^{p+\nu-1}\right).
\end{align*}
Inserting $x_S^*$ into the first part, and using the condition $\|x_S^k-x_S^*\|\leq (1/4)\|x_{S-1}-x_{S-1}^*\|,$ for all $k\geq N_S,$
we have
\begin{align*}
      &\,\,\,2^{p+\nu-2}\sigma_S\|x_S^k-x_{S-1}\|^{p+\nu-1}\notag\\
      &\,\,\leq  4^{p+\nu-1}\sigma_S\|x_S^k-x_S^*\|^{p+\nu-1}+4^{p+\nu-1}\sigma_S\|x_S^*-x_{S-1}\|^{p+\nu-1}\notag\\
     &\overset{\eqref{eqn:consecutive-sol-dis}}{\leq} 4^{p+\nu-1}\sigma_S\|x_S^k-x_S^*\|^{p+\nu-1}+4^{p+\nu-1}\sigma_S\|x_{S-1}^*-x_{S-1}\|^{p+\nu-1}\leq\tfrac{2\sigma_SD^{p+\nu-1}}{4^{(p+\nu-1)(S-2)}},
 \end{align*}
Using \eqref{eqn:quadratic-relation}, and the condition $\|x_s-x_s^*\|\leq\,(1/4)\|x_{s-1}-x_{s-1}^*\|,$
 we have
\begin{align*}
&\textstyle\sum_{i=1}^S(\sigma_i-\sigma_{i-1})\big\|\textstyle\sum_{k=i}^{S-1}(x_k-x_{k-1})\big\|^{p+\nu-1}\notag\\
% &\overset{\text{Lem}.\ref{lem:epoch-relation}}{\leq} \textstyle\sum_{i=1}^S(\sigma_i-\sigma_{i-1})\left(2\textstyle\sum_{k=i}^{S-1}\left(\left\|y_k-y_k^*\right\|^2+\left\|y_{k-1}^*-y_{k-1}\right\|^2\right)+8\textstyle\sum_{\ell>k\geq i}^{S-1}\|y_{k-1}^*-y_{k-1}\|\|y_{\ell-1}^*-y_{\ell-1}\|\right)^{\frac{p+\nu-1}{2}}\notag\\
&\leq 
\textstyle\sum_{i=1}^S(\sigma_i-\sigma_{i-1})\bigg[\textstyle\sum_{k=i}^{S-1}\tfrac{4D^2}{16^{k-1}}+\textstyle\sum_{\ell>k\geq i}^{S-1}\tfrac{8D^2}{4^{k+\ell-2}}\bigg]^{\frac{p+\nu-1}{2}}\notag\\
%   &\,\,\,\,\leq 
% \textstyle\sum_{i=1}^S(\sigma_i-\sigma_{i-1})D^{p+\nu-1}\left[2\left(\frac{\frac{1}{16^i}}{1-\frac{1}{16}}+\frac{\frac{1}{16^{i-1}}}{1-\frac{1}{16}}\right)+8\textstyle\sum_{k\geq i}^{S-1}\frac{1}{4^{k-1}}\frac{\frac{1}{4^{k}}}{1-\frac{1}{4}}\right]^{\frac{p+\nu-1}{2}}\notag\\
% &\,\,\,\,\leq 
% \textstyle\sum_{i=1}^S(\sigma_i-\sigma_{i-1})D^{p+\nu-1}\left[32\left(\frac{\frac{17}{16^{i}}}{15}\right)+32\textstyle\sum_{k\geq i}^{S-1}\frac{1}{16^{k}}\frac{4}{3}\right]^{\frac{p+\nu-1}{2}}\notag\\
% &\,\,\,\,\leq 
% \textstyle\sum_{i=1}^S(\sigma_i-\sigma_{i-1})D^{p+\nu-1}\left[32\left(\frac{\frac{17}{16^{i}}}{15}\right)+32\frac{\frac{16}{16^i}}{15}\frac{4}{3}\right]^{\frac{p+\nu-1}{2}}\notag\\
&\leq (9D)^{p+\nu-1}
\textstyle\sum_{i=1}^S\tfrac{\sigma_i-\sigma_{i-1}}{4^{i(p+\nu-1)}},
 \end{align*}
 which concludes the proof.
\end{proof}
In the following, we specify Assumptions needed to satisfy the linear convergence in terms of the distance to the optimal solution across epochs, i.e.,
$\|x_s-x_s^*\|\leq\,(1/4)\|x_{s-1}-x_{s-1}^*\|,$  
and show the convergence rate of \Cref{alg:main-composite}.
 \subsection{Extension from accelerated cubic Newton to high-order tensor methods}\label{sec:Minimization of convex Hölder smooth functions}
In this subsection, we study AR with tensor methods as subroutines to handle convex  functions with 
  $p$-th derivatives that are $\nu$-Hölder continuous. We make the following assumption regarding the subroutine $\mathcal{A}_s\coloneqq\mathcal{A} (f, \{\sigma_i\}_{i\leq s}, \{x_{i-1}\}_{i\leq s}).$
\begin{assumption}\label{as:stronger-ANPE-tensor}
   The approximate solution 
   $x_s$  exhibits the following
performance guarantees:  after $N_s$ iterations of the subroutine $\mathcal{A}_s,$ there holds
      \begin{align}
f_s(x_s)-f_s(x_s^*) &\leq \tfrac{{C}_{\mathcal{A}}L_{p,\nu}(g)\|x_s^*-{x_{s-1}}\|^{p+\nu}}{(N_s-1)^{p+\nu}},\quad \forall\,\,1\leq s\leq S.
\label{eqn:gradient-direct-bound-anpe-h}
   \end{align}
   For the $S$-th epoch, for $N_S\leq k\leq 2N_S,$
   there exists $v_S^k
   \in \partial f_S({x}_S^k), $
   % $v_S^k
   % \in \nabla g({x}_S^k) + \partial h({x}_S^k)+\textstyle\sum_{i=1}^S(\sigma_i-\sigma_{i-1})\|x_{S}^k-x_{i-1}\|^{p+\nu-2}(x_{S}^k-x_{i-1}), N_S\leq k\leq 2N_S,$
   such that
\begin{equation}\label{eq:6.14-assumption-h}
\min\limits_{k=N_S+1,\dots, 2N_S}\|v_S^k\|\leq \tfrac{{C}_{\mathcal{A}} L_{p,\nu}(g)\|x_S^*-x_{S-1}\|^{p+\nu-1}}{(N_s-1)^{p+\nu-1}},
   \end{equation}
   where ${C}_{\mathcal{A}}>1$ is a  universal constant that depends on $\mathcal{A}_s.$
    \end{assumption}
    The above assumption holds for a wide range of algorithms with different {$(p, \nu).$} See the end of this subsection for applications. We next demonstrate that with \Cref{as:stronger-ANPE-tensor} and properly chosen  epoch lengths $\{N_s\}_{s\in[S]},$ we can achieve the linear convergence in terms of the distance to the optimal solution across epochs, i.e.,
$\|x_s-x_s^*\|\leq\,(1/4)\|x_{s-1}-x_{s-1}^*\|.$  
\begin{lemma}\label{lem: linear convergence-h}
    Suppose $p+\nu\geq 2,$ and for all $s\in[S],$ the epoch length satisfies \begin{align}\label{eqn:cubic-samplesizeoptimal-h}
  N_s\geq \left\lceil4\left(\tfrac{ 2^{{p+\nu-2}}(p+\nu){C}_{\mathcal{A}}L_{p,\nu}(g)}{{\sigma_{s}}}\right)^{\frac{1}{p+\nu}}\right\rceil+1,
 \end{align}
 where $D$ is defined in \eqref{eqn:initial-optimality-gap}.
 If the subroutine $\mathcal{A}$ for solving subproblems \eqref{eqn:sub-prox-composite}
 satisfies 
 \Cref{as:stronger-ANPE-tensor}, then, for all $s\in[S],$
 \Cref{alg:main-composite} can compute $x_s$  such that 
 \begin{equation}
     \|x_s-x_s^*\|\leq\tfrac{\|x_{s-1}-x_{s-1}^*\|}{4}.
 \end{equation}
 Furthermore, in the last epoch $S,$ for all $N_S\leq k\leq 2N_S,$ we have $ \|x_S^k-x_S^*\|\leq {\|x_{S-1}-x_{S-1}^*\|}/{4}.$
\end{lemma}
\begin{proof}
   By the uniform convexity of $f_s$ and the optimality condition $\langle v_s, x_s-x_s^*\rangle\geq 0,$ where $v_s\in\partial f_s(x_s),$  we have
\begin{align}\label{eqn:uniform-convexity-fs}
    f_s(x_s)-f_s(x_s^*)
    % &\geq \langle f_s'(x_s^*), x_s-x_s^*\rangle +\frac{\sigma_s}{p+\nu}\left(\frac{1}{2}\right)^{{p+\nu-2}}\|x_s-x_s^*\|^{p+\nu}\notag\\
    &\geq \tfrac{\sigma_s}{p+\nu}\left(\tfrac{1}{2}\right)^{p+\nu-2}\|x_s-x_s^*\|^{p+\nu}, \quad\forall s\in[S].
\end{align}
Combining this with \eqref{eqn:gradient-direct-bound-anpe-h} in \Cref{as:stronger-ANPE-tensor} and \eqref{eqn:consecutive-sol-dis} in \Cref{lem:linear}, we have 
\begin{align}\label{eqn:linear-per-epoch-function-h}
\tfrac{\sigma_s}{p+\nu}\left(\tfrac{1}{2}\right)^{p+\nu-2}\|x_s-x_s^*\|^{p+\nu}
% &\leq f_s(x_s)-f_s(x_s^*)\notag\\
&\overset{\eqref{eqn:gradient-direct-bound-anpe-h},\eqref{eqn:consecutive-sol-dis}}{\leq} 
% \frac{{C}_{\mathcal{A}}L_{p,\nu}(g)\|x_s^*-{x_{s-1}}\|^{p+\nu}}{(N_s-1)^{p+\nu}}\notag\\
% &\overset{\eqref{eqn:consecutive-sol-dis}}{\leq}
\tfrac{{C}_{\mathcal{A}}L_{p,\nu}(g)\|x_{s-1}^*-{x_{s-1}}\|^{p+\nu}}{(N_s-1)^{p+\nu}}.
% &\overset{\eqref{eqn:uniform-convexity-fs}}{\leq} \frac{{C}_{\mathcal{A}}L_{p,\nu}(g)}{(N_s-1)^{p+\nu}}\left[ f_{s-1}(x_{s-1})-f_{s-1}(x_{s-1}^*)\right]\frac{p+\nu}{\sigma_{s-1}}2^{p+\nu-2}.
 \end{align}
 Therefore, by the epoch length choice in \eqref{eqn:cubic-samplesizeoptimal-h}, we have
\begin{align*}
     \|x_s-x_s^*\|^{p+\nu}&\overset{\eqref{eqn:linear-per-epoch-function-h}}{\leq}\tfrac{[{C}_{\mathcal{A}}L_{p,\nu}(g)(p+\nu)]2^{{p+\nu-2}}}{(N_s-1)^{p+\nu}\sigma_{s}}\|x_{s-1}^*-x_{s-1}\|^{p+\nu}
     % =\left(\frac{c_{\mathcal{A}}L_{p+1}(g)(p+1)2^{{\,p-1}}}{N_s^{\frac{3p+1}{2}}\sigma_{s}}\right)^{\frac{1}{p+1}}\|x_{s-1}^*-x_{s-1}\|
     \overset{\eqref{eqn:cubic-samplesizeoptimal-h}}
     % \notag\\
     % &\,{\leq}\,\frac{1}{2}\left(\frac{1}{2}\right)^{\frac{1}{p+\nu}}\left(\frac{\sigma_{s-1}}{\sigma_s}\right)^{\frac{1}{p+\nu}}\|x_{s-1}^*-x_{s-1}\|
     \leq \tfrac{\|x_{s-1}^*-x_{s-1}\|^{p+\nu}}{4^{p+\nu}}.
 \end{align*}
%   Furthermore, for the function-value residual, we have
% \begin{align*}
%       f_s(x_s)-f_s(x_s^*)&\leq\frac{{C}_{\mathcal{A}}L_{p,\nu}(g)}{(N_s-1)^{p+\nu}}\left[ f_{s-1}(x_{s-1})-f_{s-1}(x_{s-1}^*)\right]\frac{p+\nu}{\sigma_{s-1}}2^{p+\nu-2}\notag\\
%       &\overset{\eqref{eqn:cubic-samplesizeoptimal-h}}{\leq} 
%       % \frac{L_{p+1}(g)c_{\mathcal{A}}}{\frac{ 8^{p}(p+1)L_{p+1}(g)c_{\mathcal{A}}}{{\sigma_{s-1}}}}\left[ f_{s-1}(x_{s-1})-f_{s-1}(x_{s-1}^*)\right]\frac{p+1}{\sigma_{s-1}}2^{p-1}\leq
%       \frac{1}{{2^{p+\nu+1}}{{}}}\left[ f_{s-1}(x_{s-1})-f_{s-1}(x_{s-1}^*)\right].
%   \end{align*}
Similarly, notice that in the last epoch, for all $N_S\leq k\leq 2N_S,$ we have 
  \begin{align*}
     &\tfrac{\sigma_S}{p+\nu}\left(\tfrac{1}{2}\right)^{p+\nu-2}\|x_S^k-x_S^*\|^{p+\nu}\leq   f_S(x_S^k)-f_S(x_S^*)\overset{\eqref{eqn:gradient-direct-bound-anpe-h}}{\leq} \tfrac{{C}_{\mathcal{A}}L_{p,\nu}(g)\|x_S^*-{x_{S-1}}\|^{p+\nu}}{(k-1)^{p+\nu}}\notag\\
&\overset{k\geq N_S}{\leq} \tfrac{{C}_{\mathcal{A}}L_{p,\nu}(g)\|x_S^*-{x_{S-1}}\|^{p+\nu}}{(N_S-1)^{p+\nu}}\overset{\eqref{eqn:consecutive-sol-dis}}{\leq} \tfrac{{C}_{\mathcal{A}}L_{p,\nu}(g)\|x_{S-1}^*-{x_{S-1}}\|^{p+\nu}}{(N_S-1)^{p+\nu}}.
  \end{align*}
  Substituting  the epoch length  condition \eqref{eqn:cubic-samplesizeoptimal-h} concludes the proof.
  % \begin{align*}
  %      \|x_S^k-x_S^*\|&\leq\frac{({C}_{\mathcal{A}}L_{p,\nu}(g))^{\frac{1}{p+\nu}}(p+\nu)^{\frac{1}{p+\nu}}2^{\frac{p+\nu-2}{p+\nu}}}{(N_s-1)\sigma_{s}^{\frac{1}{p+1}}}\|x_{S-1}^*-x_{S-1}\|\leq \frac{1}{4}\|x_{S-1}^*-x_{S-1}\|.
  % \end{align*}
  % \begin{align*}
  %      f_S(x_S^k)-f_S(x_S^*)\leq \frac{1}{{ 2^{3p+3}}{{}}}\left[ f_{S-1}(x_{S-1})-f_{S-1}(x_{S-1}^*)\right].
  % \end{align*}
\end{proof}
Denote $p$-th-order oracle at point $x$ as $\{ f(x), \nabla f(x), \nabla^2 f(x),\dots, \nabla^p f(x)\}, p\geq 1.$
The following theorem states the main convergence results of the high-order AR framework with subroutine $\mathcal{A}$
applied for convex composite problem \eqref{eqn:composite-problem}. 
\begin{theorem}\label{thm:p-order-gradient-optimal-h}
Suppose the assumptions in \Cref{lem: linear convergence-h}, and
\begin{align}\label{eqn:p-th-number-epochs-h}
S\coloneqq\left\lceil\log_{2^{p+\nu-1}}\tfrac{{C}_{\mathcal{A}}4^{(p+\nu-2)}D^{{p+\nu-1}} L_{p,\nu}(g)}{\varepsilon}\right\rceil+1,\,\,\sigma_s\coloneqq\tfrac{2^{(p+\nu-1)(s-1)}\varepsilon}{{C}_{\mathcal{A}} 4^{(p+\nu-2)}D^{p+\nu-1}},
 \end{align}
 for all $s\in[S].$
 Then \Cref{alg:main-composite} can compute an approximate solution  $x_S^k$ such that $\min\limits_{k=N_S+1,\dots, 2N_s}\|\nu_S^k\|\leq \varepsilon,$ \text{where} $\nu_S^k\in \partial f({x}_S^k)$
after at most 
\begin{align}\label{eqn:total-complexity-h}
% \tfrac{16\left[2(p+\nu){C}_{\mathcal{A}}L_{p,\nu}(g)D^{p+\nu-1}\right]^{\frac{1}{p+\nu}}}
% {\left(1 - 2^{-1 + \frac{1}{p+\nu}}\right) \varepsilon^{\frac{1}{p+\nu}}}=
\mathcal{O}\left(\tfrac{L_{p,\nu}(g)^{\frac{1}{p+\nu}} D^{\frac{p+\nu-1}{p+\nu}}}{\varepsilon^{\frac{1}{p+\nu}}}\right)
 \end{align}
calls to the $p$-th-order oracle.
    \end{theorem}
\begin{proof}
Recall that the subgradient norm $\|\nu_S^k\|$ can be decomposed  as follows.
\begin{align*}
\min\limits_{k=N_S+1,\dots, 2N_S}\|\nu_S^k\|&\leq \underbrace{\min\limits_{k=N_S+1,\dots, 2N_S}\|v_S^k\|}_{\texttt{Term I}}+\underbrace{\textstyle\sum_{i=1}^S(\sigma_i-\sigma_{i-1})\|x_S^k-{x}_{i-1}\|^{p+\nu-1}}_{\texttt{Term II}},
 \end{align*} 
 Substituting the convergence rate of the subgradient \texttt{Term I} from 
 \Cref{as:stronger-ANPE-tensor} and utilizing \eqref{eqn:consecutive-sol-dis}, we have
\begin{align*}
(\texttt{Term I})^{p+\nu}
% &=\min\limits_{k=N_S+1,\dots, 2N_S}\|v_S^k\|\notag\\
&\,\overset{\eqref{eq:6.14-assumption-h},\eqref{eqn:consecutive-sol-dis}}{\leq} \left[\tfrac{{C}_{\mathcal{A}} L_{p,\nu}(g)\|x_{S-1}^*-x_{S-1}\|^{p+\nu-1}}{(N_S-1)^{p+\nu-1}}\right]^{p+\nu}
\notag\\
 &\quad\overset{\eqref{eqn:cubic-samplesizeoptimal-h}}{\leq} {C}_{\mathcal{A}} L_{p,\nu}(g)\left(\tfrac{ {C}_{\mathcal{A}}\sigma_{S}\|x_{S-1}^*-x_{S-1}\|^{p+\nu}}{{ {C}_{\mathcal{A}}4^{p+\nu}2^{{p+\nu-2}}(p+\nu)}{}}\right)^{p+\nu-1}\notag\\
 &\,\,\,\quad\overset{\text{(i)}}{\leq} {C}_{\mathcal{A}} L_{p,\nu}(g)\left(\tfrac{ {C}_{\mathcal{A}} \sigma_{S} D^{p+\nu}}{{ {C}_{\mathcal{A}}2^{{p+\nu-2}}(p+\nu)}{ 4^{(p+\nu)S}}}\right)^{p+\nu-1},
% &\overset{\eqref{eqn:cubic-samplesizeoptimal-h}}{\leq} {\left(\frac{9^{{p+\nu}}(p+\nu)2^{S(p+\nu-1)}\varepsilon}{3D^p}\right)}^{\frac{p+\nu-1}{p+\nu}}\frac{{C}_{\mathcal{A}} L_{p,\nu}(g)\|x_{S-1}^*-x_{S-1}\|^{p+\nu-1}}{[2^{{2p+2\nu-1}}(p+\nu){C}_{\mathcal{A}} L_{p,\nu}(g)]^{\frac{p+\nu-1}{p+\nu}}}\notag\\
\end{align*}
where in (i), we substitute the linear convergence of $\|x_s-x_s^*\|$ from \Cref{lem: linear convergence-h}. Furthermore, by 
the choice of $S$ from \eqref{eqn:p-th-number-epochs-h}, we have $\sigma_S\geq   L_{p,\nu}(g),$ therefore,
\begin{align*}
    (\texttt{Term I})^{p+\nu}&\leq  {C}_{\mathcal{A}} \sigma_S\left(\tfrac{ \sigma_{S} D^{p+\nu}}{{ 2^{{p+\nu-2}}(p+\nu)}{ 4^{(p+\nu)S}}}\right)^{p+\nu-1}\leq {C}_{\mathcal{A}} \sigma_S\left(\tfrac{ \sigma_{S} D^{p+\nu}}{{ 2^{{p+\nu-1}}}{ 4^{(p+\nu)S}}}\right)^{p+\nu-1}.
    \end{align*}
    Substituting the choice for $\sigma_S$ \eqref{eqn:p-th-number-epochs-h},  we have $\texttt{Term I}\leq \varepsilon /2.$
It remains to bound \texttt{Term II}. By \Cref{lem: linear convergence-h}, the conditions of \Cref{lem:linear} holds, thus, we have
\begin{align*}
\texttt{Term II}
&\,\,\,\,\leq\tfrac{2\sigma_SD^{p+\nu-1}}{4^{(p+\nu-1)(S-2)}}+\tfrac{(18D)^{p+\nu-1}}{2}
\textstyle\sum_{i=1}^S\tfrac{\sigma_i-\sigma_{i-1}}{4^{i(p+\nu-1)}}\notag\\
&\,\,\overset{\eqref{eqn:p-th-number-epochs-h}}{\leq}\tfrac{2\times 2^{(p+\nu-1)(S-1)}\varepsilon}{4^{(p+\nu-1)(S-1)}}+
\textstyle\sum_{i=1}^S\tfrac{[2^{(p+\nu-1)(i-1)}]18^{p+\nu-1}\varepsilon}{2\times4^{(i+1)(p+\nu-1)}}{}{}
% \notag\\
% &\,\,\,\,
% \leq\frac{2\times 4^{(p+\nu-2)}\varepsilon}{4^{(p+\nu-2)(S-1)}}+
% \textstyle\sum_{i=1}^S\frac{3\times 4^{i-1}\varepsilon}{32}\frac{18^{p+\nu-1}}{4^{i(p+\nu-1)}}
\leq\tfrac{\varepsilon}{2}.
 \end{align*}
 Combining the bounds for \texttt{Term I, II}, we have $\|\nu_s\|\leq \varepsilon.$ The total calls to the iterations of the subroutines is bounded by 
\begin{align*}
    \textstyle\sum_{s=1}^S N_s&\overset{\eqref{eqn:cubic-samplesizeoptimal-h}}
% {=}\textstyle\sum_{s=1}^S\tfrac{ 4\left[2^{{p+\nu-1}}(p+\nu){C}_{\mathcal{A}}L_{p,\nu}(g)\right]^{\frac{1}{p+\nu}}}{\left(\frac{2^{(p+\nu-1)(s-1)}\varepsilon}{4^{(p+\nu-2)}{C}_{\mathcal{A}} D^{p+\nu-1}}\right)^{\frac{1}{p+\nu}}}+S
    \leq\tfrac{16\left[2(p+\nu){C}_{\mathcal{A}}L_{p,\nu}(g)D^{p+\nu-1}\right]^{\frac{1}{p+\nu}}}{{\left(1 - 2^{-1 + \frac{1}{p+\nu}}\right) \varepsilon^{\frac{1}{p+\nu}}}}+S.
%     \overset{\eqref{eqn:cubic-samplesizeoptimal-h}}{=}\mathcal{O}\left(
% \frac{D^{\frac{p+\nu-1}{p+\nu}}}{\varepsilon^{\frac{1}{p+\nu}}}\right).
\end{align*}
% Direct calculation concludes the proof.
  \end{proof}
 Similar to \Cref{ACN}, the regularization parameter $\sigma_s$ starts at $\mathcal{O}(\varepsilon/D^{p+\nu-1})$ and is
exponentially increasing, while the epoch length is exponentially decreasing, scaling inversely with respect to the regularization parameter.
\Cref{alg:main-composite} terminates when $\sigma_S$ has reached $L_{p,\nu}(g).$
Other choices of the regularization can also be used, with the number of epochs changing accordingly. 

By \Cref{thm:p-order-gradient-optimal-h}, we can accelerate the convergence rate of the gradient norm by employing various tensor-based methods as subroutines. Examples include the accelerated cubic Newton method~\cite{nesterov2008accelerating}, accelerated tensor methods \cite{nesterov2021implementable},
 the bi-level unconstrained minimization approach~\cite{nesterov2023inexact}, and accelerated proximal-point methods with lower-order solvers~\cite{nesterov2021superfast,nesterov2023inexact}. Notably, in \cite{nesterov2021superfast,nesterov2023inexact}, they show that second-order methods can achieve the same convergence rate as third-order methods in terms of the function-value residual. As a result, using such second-order methods as subroutines in the high-order AR framework \Cref{alg:main-composite} allows the convergence rate in terms of the gradient norm to match that of the function-value residual in \cite{nesterov2021superfast,nesterov2023inexact}. In what follows, we present the Accelerated Regularized Tensor Method (ARTM)~\cite{grapiglia2022tensor,grapiglia2020tensor} as an inner subroutine for solving~\eqref{eqn:composite-problem} with $\nu \in [0\, 1]$, which improves the gradient complexity bound of ARTM  
in~\citep[Theorem 5.7]{grapiglia2022tensor} by removing a logarithmic factor. The gradient complexities for the other methods follow by the same reasoning.
\begin{corollary}
    Suppose the subroutine $\mathcal{A}$ is  ARTM. Then, \Cref{alg:main-composite} can compute an approximate solution $x_S^k$ s.t. $\min_{k = N_S + 1, \dots, 2N_S} \|\nu_S^k\| \leq \varepsilon$, where $\nu_S^k \in \nabla g(x_S^k) + \partial h(x_S^k)$,  
within $\mathcal{O}(({D^{{p+\nu-1}{}}}/{\varepsilon})^{\frac{1}{p+\nu}}) $
calls to the $p$-th-order oracle,
where  $D$ is defined in \eqref{eqn:initial-optimality-gap}, and $p \geq 2,$ $\nu \in [0\,\, 1].$
\end{corollary}
\begin{proof}
   \citep[Theorem 5.7]{grapiglia2022tensor} shows that  
\Cref{as:stronger-ANPE-tensor} holds for ARTM,  
thus \Cref{thm:p-order-gradient-optimal-h} applies.
\end{proof}
This result is the tightest known for convex and $p$-times differentiable functions with $\nu$-Hölder continuous $p$-th derivative. Although it does not match the lower complexity bounds for such function class \citep[Theorem~6.6]{grapiglia2022tensor}, given by 
$\Omega\left(\left({D^{{p+\nu-1}}}/{\varepsilon}\right)^{\frac{2}{3(p+\nu)-2}}\right).$
The gap arises because no existing methods have yet achieved the optimal rate in terms of the function-value residual for such function class. 
If such methods are developed, the high-order AR framework can incorporate them as subroutines $\mathcal{A}$ to accelerate the gradient-norm convergence and achieve a matching rate.

  \subsection{Fast gradient minimization for convex smooth function}\label{Sec:Optimal gradient minimization}
In this subsection, we study AR with optimal $p$-th-order methods as subroutines to handle composite convex  functions with 
Lipschitz continuous $p$-th derivatives.
 Specifically,  there exists \( L_{p+1}(g) > 0 \) s.t.
\begin{align}\label{assumption:L-p}
    \|D^p g(x)-D^p g(y)\|  \leq L_{p+1}(g)\|x-y\|,\quad \forall\,\, x, y\in\mathbb{R}^n.
\end{align}
We make the following assumption regarding the subroutine $\mathcal{A}_s.$
\begin{assumption}\label{as:stronger-ANPE}
   The approximate solution 
   $x_s$  exhibits the following
performance guarantees:  after $N_s$ iterations of the subroutine $\mathcal{A}_s,$ there holds
      \begin{align}
f_s({x_s})-f_s(x_s^*) &\leq \tfrac{c_{\mathcal{A}}L_{p+1}(g)\|x_s^*-{x_{s-1}}\|^{p+1}}{N_s^{\frac{3p+1}{2}}},\quad \forall\,\,1\leq s\leq S.
\label{eqn:gradient-direct-bound-anpe}
   \end{align}
 For the $S$-th epoch, for $N_S\leq k\leq 2N_S,$ 
   there exists $v_S^k
   \in \partial f_S({x}_S^k),$ such that
\begin{equation}\label{eq:6.14-assumption}
\min\limits_{k=N_S+1,\dots, 2N_S}\|v_S^k\|\leq\tfrac{{c}_{\mathcal{A}}L_{p+1}(g)\|x_{S}^*-x_{S-1}\|^{p}}{{N_S^{\frac{3p}{2}}}},
   \end{equation}
   where $c_{\mathcal{A}}$ is universal constants that depend on $\mathcal{A}_s.$
    \end{assumption}
    Compared with \Cref{as:stronger-ANPE-tensor}, \Cref{as:stronger-ANPE} imposes faster rates for both the function gap and the gradient norm. Note, however, that the gradient-norm rate is still slower than the function-value rate.
  The above assumption holds for a wide range of optimal algorithms with different $p.$ The subroutines examples include first-order algorithms: the accelerated gradient descent method (AGD) \cite{nesterov2018lectures}, second-order algorithms: large stepsize accelerated hybrid proximal extragradient \cite{monteiro2013accelerated}, and $p$-th-order algorithms: 
optimal tensor method (OTM) \cite{bubeck2019near, jiang2021optimal,gasnikov2019optimal} or  inexact $p$-th-order proximal point method with tensor step \cite{nesterov2021inexact,nesterov2023inexact}.

  We next show that with properly chosen  epoch lengths $\{N_s\}_{s\in[S]}$ and \Cref{as:stronger-ANPE}, high-order AR \Cref{alg:main-composite} can achieve the linear convergence in terms of distance to the optimal solution across epochs. The proof is similar to \Cref{lem: linear convergence-h}, thus omitted for simplicity.
\begin{lemma}\label{lem: linear convergence}
    Suppose $p\geq 1,$ and for all $s\in[S],$ the epoch length satisfies 
\begin{align}\label{eqn:cubic-samplesizeoptimal}
  N_s&\geq\left\lceil4\left(\tfrac{ (p+1)c_{\mathcal{A}}L_{p+1}(g)}{{\sigma_{s}}}\right)^{\frac{2}{3p+1}}\right\rceil,
 \end{align}
 where $D$ is defined in \eqref{eqn:initial-optimality-gap}.
 If the subroutine for solving subproblems satisfies \Cref{as:stronger-ANPE}. Then, for all $s\in[S],$
 \Cref{alg:main-composite} can compute $x_s$  such that $\|x_s-x_s^*\|\leq{\|x_{s-1}-x_{s-1}^*\|}/{4}.$
 Furthermore, in the last epoch $S,$ for all $N_S\leq k\leq 2N_S,$ we have $ \|x_S^k-x_S^*\|\leq {\|x_{S-1}-x_{S-1}^*\|}/{4}.$
\end{lemma}
  With \Cref{lem: linear convergence} in place, we have the following convergence results of the \Cref{alg:main-composite},
the proof is similar to \Cref{thm:p-order-gradient-optimal-h}, and thus we omit it for simplicity.
\begin{theorem}\label{thm:p-order-gradient-optimal}
Suppose the assumptions in \Cref{lem: linear convergence}, and \begin{align}\label{eqn:p-th-number-epochs}
S&\coloneqq\left\lceil\tfrac{2}{3p+1}\log_2\tfrac{{{C}_{\mathcal{A}} D^{p}L_{p+1}(g)}}{\varepsilon}\right\rceil+1,\quad \sigma_s\coloneqq\tfrac{2^{\frac{(s-1)(3p+1)}{2}}\varepsilon }{{C}_{\mathcal{A}} D^{p}},
% ,\quad S=\left\lceil\log_8\frac{64LD^2}{p\varepsilon}\right\rceil
 \end{align} 
 for all $s\in[S].$
  Then \Cref{alg:main-composite} can compute an approximate solution $x_S^k$ together with $\nu_S^k,$
  such that $\nu_S^k\in \partial f(x_S^k),$
  and $\min\limits_{k=N_S+1,\dots, 2N_S}\|\nu_S^k\|\leq \varepsilon$ 
after at most 
\begin{align}\label{eqn:total-complexity}
% \tfrac{8\left[(p+1)L_{p+1}(g)D^p\right]^{\frac{2}{3p+1}}}{\varepsilon^{\frac{2}{3p+1}}}
% =
\mathcal{O}\left(\left[\tfrac{L_{p+1}(g)D^p}{\varepsilon}\right]^{\frac{2}{3p+1}}\right)
 \end{align}
calls to the $p$-th-order oracle.
    \end{theorem}
    
 We can utilize \Cref{thm:p-order-gradient-optimal} to derive optimal algorithms in terms of  (sub)gradient norm minimization 
for convex composite problem \eqref{eqn:composite-problem}
by employing various tensor-based methods as subroutines. We illustrate with the following examples.
\begin{corollary}
    Suppose the subroutine $\mathcal{A}$ is A-NPE \cite{monteiro2013accelerated}. Then, a third-order AR framework \Cref{alg:main-composite} can compute an approximate solution $x_S^k$ such that $\min_{k = N_S + 1, \dots, 2N_S} \|\nu_S^k\| \leq \varepsilon$, where $\nu_S^k \in \nabla g(x_S^k) + \partial h(x_S^k)$,  
within $$\mathcal{O}([{L_{3}(g)D^p}/{\varepsilon}]^{\frac{2}{7}} \log \tfrac{1}{\varepsilon}) $$
evaluations of the gradient and Hessian,
where  $D$ is defined in \eqref{eqn:initial-optimality-gap}.
\end{corollary}
\begin{proof}
   \citep[Theorem 4.1]{monteiro2013accelerated} shows that  
\Cref{as:stronger-ANPE} holds for A-NPE,  
thus \Cref{thm:p-order-gradient-optimal} applies.
\end{proof}

We next show a  gradient norm minimization using a second-order method for a third-order smooth problem, and show its
convergence rate of the gradient norm matches that of the function-value residual in \cite{nesterov2021inexact}.  For simplicity, we consider the unconstrained problem \eqref{eqn:main} in this example.
\begin{corollary}
 Suppose the subroutine $\mathcal{A}$ is the second order implementation of the inexact third-order proximal point method \cite{nesterov2021inexact}. Then, a fourth-order AR \Cref{alg:main-composite} can compute an approximate solution $x_S^k$ such that $$\min_{k = N_S + 1, \dots, 2N_S} \|\nabla f(x_S^k)\| \leq \varepsilon$$ 
within $\mathcal{O}([{L_{4}(f)D^3}/{\varepsilon}]^{\frac{1}{5}}) $
iterations,
where  $D$ is defined in \eqref{eqn:initial-optimality-gap}, at each iteration, it calls the second-order once and the first-order
oracle $\mathcal{O}(\log{\varepsilon^{-1}})$ times at most.
\end{corollary}
\begin{proof}
  \citep[Theorem 3.8, 5.5]{nesterov2021inexact} show that  
\Cref{as:stronger-ANPE} holds with $p=3$,  
thus \Cref{thm:p-order-gradient-optimal} applies.
\end{proof}

\smallskip

\section{Parameter-free algorithm for convex problems}\label{sec:parameter-free-convex}
All convergence results in the \Cref{sec:second-order-methods} and \Cref{sec:Optimal high-order methods for structured convex problems} are based on the exact solutions of the auxiliary optimization problem at each iteration of the high-order method, and assume access to certain key problem parameters, including the Lipschitz constant of the $\nu$-Hölder continuous $p$-th derivative $L_{p,\nu}(f)$, and an upper bound $D$ on the initial distance to the solution set, such that $\operatorname{dist}(x_0,X^*) \le D$. However, in general, the auxiliary subproblem at each iteration cannot be solved exactly, and those problem-dependent parameters are not known beforehand.
In this section, we investigate inexact and parameter-free implementations of high-order methods for gradient minimization that do not rely on such information. For clarity of exposition, we first focus on the unconstrained setting; analogous parameter-free strategies can be developed for problems with simple constraints or composite structure.

Denote 
$(s, k)$ as the pair referring to the $k$-th iteration of epoch $s,$ and 
$\mathcal{A}_s\coloneqq\mathcal{A} (f, \{\sigma_i\}_{i\leq s}, \{x_{i-1}\}_{i\leq s}).$
 \begin{algorithm}[!ht]
\caption{High-order AR  without knowledge of  $L_{p,\nu}(f)$}\label{alg:main-composite-pf}
\begin{algorithmic}[1]
 \Require Initial regularization $\sigma_0=0, \sigma_1>0;$ initial point $x_0\in\mathbb{R}^n;$ initial line-search value $L_0.$
  \For{$s=1,2,\dots$} 
  \State If $s>1,$ set $\sigma_s=2^{p+\nu-1}\sigma_{s-1}.$
    \State 
    Compute an approximate solution $x_s$   of $f_s(x)$ by running $\mathcal{A} _s$ with the  initialization $(x_{s-1}, L_{s-1}).$
\begin{align}\label{eqn:sub-prox-composite-pf}
 x_s\approx \argmin\limits_{x\in \mathbb{R}^n} \left\{ f_s(x)\coloneqq f(x)+\textstyle\sum_{i=1}^s\tfrac{\sigma_i-\sigma_{i-1}}{p+\nu}\|x-{x}_{i-1}\|^{p+\nu}\right\},
    \end{align}
\State and compute $L_s=L_{{s-1},{N_{s-1}}},$ the line-search value at $(s-1, N_{s-1}).$
Here $p\geq 1,  \nu\in[0,1],$ and $p+\nu\geq 2.$ 
% and initial line-search value $L_{s-1}\coloneqq L_{{s-1},{N_{s-1}}}.$ Here $L_{{s-1},{N_{s-1}}}$ is the line-search value at $(s-1, N_{s-1}).$
 % Let  $\mathcal{A}_s$  choose an iteration index $N_s$ as $k,$ and output the 
 %   $k$-th iterate as 
 %    $x_s,$  its line-search value $L_{s, k}$ when
 Let $k$ be the smallest integer satisfying
 \begin{equation*}
    k\geq 8\left[\tfrac{L_{{s},{k}}(p+\nu)}{4\sigma_s}\right]^{\frac{1}{p+\nu}}+1.
    \end{equation*}
    \State Set $N_s\coloneqq k$, and output  $x_s=x_{s,N_s}$ with  its line-search value $L_{s, N_s}.$ 
    \State Run an additional $N_s$
iterations to record 
    $L_{s,{2N_s}}.$
        \State If $\sigma_s\geq \tfrac{(L_{s,{2N_s}})^{(p+\nu)^2}}{(L_{s,{N_s}})^{(p+\nu-1)(p+\nu+1)}},$ then {\bf terminate} with $\widehat{x}=\argmin\limits_{N_s<k\leq 2N_s}\{\|\nabla f(z_s^{k})\|\}.$
     \EndFor
 \State \Return{$\widehat{x}$}
\end{algorithmic}
\end{algorithm}

We make the following assumption on the subroutine $\mathcal{A}_s$ used to solve \eqref{eqn:sub-prox-composite-pf}.
\begin{assumption}\label{as:stronger-ANPE-tensor-pf}
   The subroutine $\mathcal{A}_s$   has the following performance guarantee: for any $k> 1,$ 
      \begin{align}
f_s({x_s^k})-f_s(x_s^*) &\leq \tfrac{L_{s,k}\|{x_{s-1}}-x_s^*\|^{p+\nu}}{(k-1)^{p+\nu}},\quad \forall\,\,1\leq s\leq S,
\label{eqn:gradient-direct-bound-anpe-h-pf}
   \end{align}
   where $x_s^k$ is the computed approximate solution and $L_{s,k}$ is a local estimate of Hölder constant of $f$ at $(s, k)$ such that \begin{equation}\label{eqn:estimate-line-earch-1}L_{s,k}\leq c_{\mathcal{A}} \max\{pL_{p,\nu}(f), L_0, \theta \},
   \end{equation}
   where $c_{\mathcal{A}}$
   is a universal constant  depending on the subroutine $\mathcal{A}_s,$ and $\theta$ is a user-defined subproblem inexactness parameter controlling the inner accuracy in $\mathcal{A}_s.$
   For each $s$-th epoch, for $N_s\leq k\leq 2N_s,$
   there exists sequences $\{{z}_s^k\}$ with $f({z}_s^k)\leq f(x_s^k)$ and
   $\nabla f_s({z}_s^k), $
   such that
\begin{equation}\label{eq:6.14-assumption-h-pf}
\min\limits_{k=N_s+1,\dots, 2N_s}\|\nabla f_s({z}_s^k)\|\leq \tfrac{ L_{s, 2N_s}\|x_{s-1}-x_s^*\|^{p+\nu-1}}{(2N_s-1)^{\frac{(p+\nu-1)(p+\nu+1)}{p+\nu}}},
   \end{equation}
  where $L_{s,2N_s}$ is a local estimate of Hölder constant at $(s, 2N_s)$ such that
  \begin{equation}\label{eqn:estimate-line-search-2}
      L_{s,{2N_s}}\leq c_{\mathcal{A}} \max\{pL_{p,\nu}(f), L_0, \theta\}.
  \end{equation}
    \end{assumption}
Notice that \Cref{as:stronger-ANPE-tensor-pf} relaxes \Cref{as:stronger-ANPE-tensor} in two aspects. First, it does not require knowledge of $L_{p,\nu}(f);$ instead, it uses local estimate $L_{s,k}.$ Second, it allows inexact inner solution, quantified by $\theta,$ where smaller $\theta$ means a more accurate solution for each iteration.
This assumption can be easily satisfied by many adaptive algorithms, for example accelerated tensor methods \cite{grapiglia2020tensor,grapiglia2022tensor,grapiglia2023adaptive,He2025HAR}.

\begin{theorem}\label{thm:p-order-gradient-optimal-h-pf}
Suppose the subroutine satisfies \Cref{as:stronger-ANPE-tensor-pf}.
Given $x_0, L_0$ and $\sigma_1,$ and suppose
$\sigma_1\leq c_{\mathcal{A}}.$
Furthermore, choose the regularization parameter $ \sigma_s=2^{p+\nu-1}\sigma_{s-1},$ for all $s\geq 2.$ Suppose
the epoch length $N_s$ as
\begin{align}\label{eqn:sample-size-pf}
 N_s\geq 8\left[\tfrac{L_{s,N_s}(p+\nu)}{4\sigma_s}\right]^{\frac{1}{p+\nu}}+1,\quad\forall s\geq1.
\end{align}
 Then \Cref{alg:main-composite-pf} can compute an approximate solution  $\widehat{x}$ such that
 \begin{align}\label{eqn:gradient-sigma}
     \|\nabla f(\widehat{x})\|\leq 3\sigma_1[9\cdot\textnormal{dist}(x_0, X^*)]^{p+\nu-1}
 \end{align}
after $\mathcal{O}\left(\left[\tfrac{\max\{{pL_{p,\nu}}(f), L_0, \theta \}}{\sigma_{1}}\right]^{\frac{1}{p+\nu}}\right)$
{calls to the $p$-th-order oracle.}
    \end{theorem}

\begin{proof}
Observe that the gradient norm $\|\nabla f(z_S^k)\|$ can be decomposed  as follows.
\begin{align*}
&\min\limits_{k=N_S+1,\dots, 2N_S}\|\nabla f(z_S^k)\|\notag\\
&\leq \underbrace{\min\limits_{k=N_S+1,\dots, 2N_S}\|\nabla f_S(z_S^k)\|}_{\texttt{Term I}}+\underbrace{\textstyle\sum_{i=1}^S(\sigma_i-\sigma_{i-1})\|x_S^k-{x}_{i-1}\|^{p+\nu-1}}_{\texttt{Term II}},
 \end{align*} 
We start with bounding \texttt{Term II}.  By the uniform convexity of $f_s$ and the optimality condition $\nabla f_s(x_s^*)= 0,$ we have
\begin{align}\label{eqn:uniform-convexity-fs-pf}
    f_s(x_s)-f_s(x_s^*)
    &\geq \tfrac{\sigma_s}{p+\nu}\left(\tfrac{1}{2}\right)^{p+\nu-2}\|x_s-x_s^*\|^{p+\nu}, \quad\forall s\in[S].
\end{align}
Similar to the proof of \Cref{ACN}, by using the optimality condition of \eqref{eqn:sub-prox-composite-pf}, for all $s\geq 1,$ $s\in[S],$ we have
\begin{align}
        \|x_{s-1}-x_s^*\|&\leq \|x_{s-1}-x_{s-1}^*\|.\label{eqn:consecutive-sol-dis-pf}
        \end{align}
Combining this with \eqref{eqn:gradient-direct-bound-anpe-h-pf} in \Cref{as:stronger-ANPE-tensor-pf}, we have 
\begin{align}\label{eqn:linear-per-epoch-function-h-pf}
\tfrac{\sigma_s}{p+\nu}\left(\tfrac{1}{2}\right)^{p+\nu-2}\|x_s-x_s^*\|^{p+\nu}
% &\leq f_s(x_s)-f_s(x_s^*)\notag\\
&\overset{\eqref{eqn:gradient-direct-bound-anpe-h-pf}}{\leq} \tfrac{L_{s,k}\|{x_{s-1}}-x_{s-1}^*\|^{p+\nu}}{(k-1)^{p+\nu}}.
 \end{align}
Combining this with the subroutine sample size condition \eqref{eqn:sample-size-pf},
we have
\begin{align}\label{eqn:linear-pf}
    \|{x_{s}}-x_{s}^*\|\leq\tfrac{1}{4} \|{x_{s-1}}-x_{s-1}^*\|.
\end{align}
Similarly, notice that in each epoch $s$,  we have 
  \begin{align*}
     \tfrac{\sigma_s}{p+\nu}\left(\tfrac{1}{2}\right)^{p+\nu-2}\|z_s^k-x_s^*\|^{p+\nu}&\,\,\leq   f_s(z_s^k)-f_s(x_s^*)\notag\\
     &\overset{\text{A.3}}{\leq}f_s(x_s^k)-f_s(x_s^*)\overset{\eqref{eqn:gradient-direct-bound-anpe-h-pf}}{\leq} \tfrac{L_{s,k}\|x_s^*-{x_{s-1}}\|^{p+\nu}}{(k-1)^{p+\nu}}\notag\\
&\overset{\eqref{eqn:consecutive-sol-dis-pf}}{\leq} \tfrac{L_{s,k}\|x_{s-1}^*-{x_{s-1}}\|^{p+\nu}}{(k-1)^{p+\nu}}.
  \end{align*}
  Combining this with the subroutine sample size condition \eqref{eqn:sample-size-pf},
we have $ \|{z_{s}^k}-x_{s}^*\|\leq\tfrac{1}{4} \|{x_{s-1}}-x_{s-1}^*\|.$
Therefore, the conditions of \Cref{lem:linear} holds, thus, we have
\begin{align*}
\texttt{Term II}
&\,\,\,\,\leq\tfrac{2\sigma_S\|x_0-x^*\|^{p+\nu-1}}{4^{(p+\nu-1)(S-2)}}+\tfrac{(18\|x_0-x^*\|)^{p+\nu-1}}{2}
\textstyle\sum_{i=1}^S\tfrac{\sigma_i-\sigma_{i-1}}{4^{i(p+\nu-1)}}\notag\\
&\,\,\overset{\text{(i)}}{\leq}\tfrac{2\cdot 2^{(p+\nu-1)(S-1)}\sigma_1 \|x_0-x^*\|^{p+\nu-1}}{4^{(p+\nu-1)(S-2)}}+\tfrac{(18\|x_0-x^*\|)^{p+\nu-1}}{2}\textstyle\sum_{i=1}^S\tfrac{2^{(p+\nu-1)(i-1)}\sigma_1}{4^{i(p+\nu-1)}}\notag\\
&\,\,\overset{\text{(ii)}}{\leq} 2\sigma_1(9\|x_0-x^*\|)^{p+\nu-1},
 \end{align*}
 where $x^*$ is any solution in the optimal solution set $X^*.$
 In (i), we substitute the choice for $\sigma_s=2^{p+\nu-1}\sigma_{s-1}.$ In (ii), we used $p+\nu\geq 2,$ and $S\geq 1.$ It remains to bound \texttt{Term I}. 
 Substituting the slow gradient norm convergence rate from 
 \Cref{as:stronger-ANPE-tensor-pf}, we have 
\begin{align*}
(\texttt{Term I})^{p+\nu}
&\overset{\eqref{eq:6.14-assumption-h-pf},\eqref{eqn:consecutive-sol-dis-pf}}{\leq} \left[\tfrac{L_{S, 2N_S}\|x_{S-1}^*-x_{S-1}\|^{p+\nu-1}}{(2N_S-1)^{\frac{(p+\nu-1)(p+\nu+1)}{p+\nu}}}\right]^{p+\nu}
\notag\\
 &\,\,\,\,\,\overset{\eqref{eqn:sample-size-pf}}{\leq} \left[\tfrac{L_{S, 2N_S}(4\sigma_S)^{\frac{(p+\nu-1)(p+\nu+1)}{(p+\nu)^2}}\|x_{S-1}^*-x_{S-1}\|^{p+\nu-1}}{16^{\frac{(p+\nu-1)(p+\nu+1)}{p+\nu}}[L_{S, N_S}(p+\nu)]^{\frac{(p+\nu-1)(p+\nu+1)}{(p+\nu)^2}}}\right]^{p+\nu}
\notag\\
 &\,\,\,\,\,\,\overset{\text{(iii)}}{\leq} \sigma_S\left(\tfrac{{16\sigma_S\|x_{S-1}^*-x_{S-1}\|^{p+\nu}}{}}{16^{p+\nu}(p+\nu)}\right)^{p+\nu-1}
\overset{\text{(iv)}}{\leq} \sigma_S\left(\tfrac{  16\sigma_{S} \|x_0-x^*\|^{p+\nu}}{{ (p+\nu)}{ 4^{(p+\nu)(S+1)}}}\right)^{p+\nu-1},
\end{align*}
where in (iii), we used the termination rule in \Cref{alg:main-composite-pf}. In (iv),
we substitute the linear convergence of $\|x_s-x_s^*\|$ from \eqref{eqn:linear-pf}. 
Substituting the choice of $\sigma_S=2^{(p+\nu-1)(S-1)}\sigma_1,$ we have
\begin{align}
    \texttt{Term I}&\leq 2^{(p+\nu-1)(S-1)}\sigma_1\left(\tfrac{  16\|x_0-x^*\|^{p+\nu}}{{ (p+\nu)}{ 4^{(p+\nu)(S+1)}}}\right)^{\frac{p+\nu-1}{p+\nu}}\leq \sigma_1\|x_0-x^*\|^{p+\nu-1}.
\end{align}
Combining the errors from \texttt{Term I, II}, we have
\begin{align*}
    \|\nabla f(\widehat{x})\|&=\min\limits_{k=N_S+1,\dots, 2N_S}\|\nabla f(z_S^k)\|\notag\\
    &\leq 2\sigma_1(9\|x_0-x^*\|)^{p+\nu-1} +\sigma_1\|x_0-x^*\|^{p+\nu-1}\leq 3\sigma_1(9\|x_0-x^*\|)^{p+\nu-1} .
\end{align*}
Furthermore, by 
the termination rule in \Cref{alg:main-composite-pf}, we have
\begin{align}
\tfrac{(L_{S,{2N_S}})^{(p+\nu)^2}}{(L_{S,{N_S}})^{(p+\nu-1)(p+\nu+1)}}&\overset{\eqref{alg:main-composite-pf}}{\leq} \sigma_S=\sigma_12^{(p+\nu-1)(S-1)}\notag\\
&\;\,=2^{p+\nu-1} \sigma_{S-1}\overset{\eqref{alg:main-composite-pf}}{\leq} 2^{p+\nu-1}\tfrac{(L_{s-1,{2N_{s-1}}})^{(p+\nu)^2}}{(L_{s-1,{N_{s-1}}})^{(p+\nu-1)(p+\nu+1)}} \notag\\
&\overset{\eqref{eqn:estimate-line-search-2}}{\leq} 2^{p+\nu-1}c_{\mathcal{A}} \max\{pL_{p,\nu}(f), L_0, \theta\}^{(p+\nu)^2}\left(\tfrac{1}{L_0}\right)^{(p+\nu-1)(p+\nu+1)}.
\end{align}
  Therefore, we have
  \begin{align*}
      S&\leq 2+\tfrac{(p+\nu)^2}{{p+\nu-1}}\log_{2}\left(\tfrac{c_{\mathcal{A}}^{\frac{1}{(p+\nu)^2}} \max\{pL_{p,\nu}(f), L_0, \theta\}}{\sigma_1^{\frac{1}{(p+\nu)^2}}}\right)-(p+\nu+1)\log_{2}L_0.
  \end{align*}
 The RHS is positive as $\sigma_1\leq c_{\mathcal{A}}.$ This is mild as $\sigma_1$ serves as the accuracy precision (small) and $c_{\mathcal{A}}$ is a subroutine-dependent constant.
The total oracle calls read
\begin{align*}
\textstyle\sum_{s=1}^S2N_s&\overset{\eqref{eqn:sample-size-pf}}
\leq16\textstyle\sum_{s=1}^S\left[\tfrac{L_{s,N_s}(p+\nu)}{4\sigma_s}\right]^{\frac{1}{p+\nu}}+S\notag\\
% &
% \leq16\textstyle\sum_{s=1}^S\left[\tfrac{c_{\mathcal{A}} \max\{L_{p,\nu}(f), L_0, \theta \}(p+\nu)}{4\sigma_s}\right]^{\frac{1}{p+\nu}}+S\notag\\
&\overset{\eqref{eqn:estimate-line-earch-1}}
\leq16\left[\tfrac{c_{\mathcal{A}} \max\{pL_{p,\nu}(f), L_0, \theta \}(p+\nu)}{4\sigma_{1}}\right]^{\frac{1}{p+\nu}}\textstyle\sum_{s=1}^S\left[\tfrac{1}{2^{(p+\nu-1)(s-1)}}\right]^{\frac{1}{p+\nu}}+S\notag\\
&\,\,
\leq32\left[\tfrac{c_{\mathcal{A}} \max\{pL_{p,\nu}(f), L_0, \theta \}(p+\nu)}{4\sigma_{1}}\right]^{\frac{1}{p+\nu}}+S.
\end{align*}
This concludes the proof.
  \end{proof}
  The expression \eqref{eqn:gradient-sigma}  clarifies the dependence of the convergence rate on the initial regularization parameter $\sigma_1.$ Specifically, the following comments are in order.
\begin{itemize}
    \item[\bf 1)] {\bf Matching the convergence rate of the function-value:}
      If we choose  $\sigma_1={\varepsilon}/{\{3[9\cdot\operatorname{dist}(x_0, X^*)]^{p+\nu-1}\}},$ then by \Cref{thm:p-order-gradient-optimal-h-pf}, \Cref{alg:main-composite-pf} can compute an approximate solution  $\widehat{x}$ such that $\|\nabla f(\widehat{x})\|\leq\varepsilon$ after 
\begin{align}\label{eqn:perfect-D}
\mathcal{O}\left(\left[\tfrac{\max\{{pL_{p,\nu}}(f), L_0, \theta \}\operatorname{dist}(x_0, X^*)^{p+\nu-1}}{\varepsilon}\right]^{\frac{1}{p+\nu}}\right)
     \end{align}
{calls to the $p$-th-order oracle.} This convergence result for inexact and adaptive gradient minimization is new and matches the faster function-value rate in the subroutine $\mathcal{A}$ (cf. \Cref{as:stronger-ANPE-tensor-pf}), see for example, \cite{grapiglia2020tensor,grapiglia2022tensor, jiang2020unified,grapiglia2023adaptive}. It improves over the current convergence result for adaptive and inexact gradient minimization as \eqref{eq:6.14-assumption-h-pf} shown in \cite{grapiglia2023adaptive}. 
 \item[\bf 2)] {\bf Unknown initial optimality gap $D$:}
However, in general $\operatorname{dist}(x_0, X^*)$ is unknown, and we may use $D$ to estimate it. If $D>\operatorname{dist}(x_0, X^*),$ by choosing
 $\sigma_1={\varepsilon}/[3(9D)^{p+\nu-1}],$ we can compute an approximate solution   $\widehat{x}$ using \Cref{alg:main-composite-pf}  such that $\|\nabla f(\widehat{x})\|\leq\varepsilon$ after $$\mathcal{O}\left(\left[\tfrac{\max\{{pL_{p,\nu}}(f), L_0, \theta \}D^{p+\nu-1}}{\varepsilon}\right]^{\frac{1}{p+\nu}}\right)$$
{calls to the $p$-th-order oracle,}
 which is worse than \eqref{eqn:perfect-D}. Next we show that we can avoid such case by choosing $D$ properly. Specifically, run $\mathcal{A}_1$ with $N_1=1,$
set $D^{p+\nu-1}\coloneqq \|\nabla f_1({x}_1^2)\|/L_{1,2},$ by \eqref{eq:6.14-assumption-h-pf}, we have
 \begin{equation*}
\|\nabla f_1({x}_1^2)\|\leq{ L_{1,2}\|x_{0}-x_1^*\|^{p+\nu-1}}\overset{\eqref{eqn:consecutive-sol-dis-pf}}{\leq}{ L_{1,2}\|x_{0}-x^*\|^{p+\nu-1}},
   \end{equation*}
 where $x_1^2$ is the computed approximate solution at epoch $1$ after $2$ iterations, and
 $x^*$ is any optimal solution in $X^*.$ Therefore, we have $D\leq \operatorname{dist}(x_0, X^*).$
 If $D<\operatorname{dist}(x_0, X^*),$
 by choosing
 $\sigma_1={\varepsilon}/[3(9D)^{p+\nu-1}],$ the computed approximate solution $\widehat{x}$ only satisfies $\|\nabla f(\widehat{x})\|\leq \varepsilon[\operatorname{dist}(x_0, X^*)/{D}]^{p+\nu-1}.$ Therefore, it does not attain the desired accuracy $\varepsilon.$ In this case, we simply discard all previous results and restart the computation with a larger guess $D,$ this increasing $D$ schedule is a ``guess-and-check'' procedure. It was first used for parameter-free optimal gradient minimization with first-order methods \cite{lan2024optimalparameterfreegradientminimization}. We will show below that such a strategy maintains the same oracle complexity as if $\operatorname{dist}(x_0,X^*)$ were known, while still guaranteeing $\|\nabla f(\widehat{x})\|\le \varepsilon$.
\end{itemize}
\begin{algorithm}[!ht]
\caption{A guess-and-check implementation of \Cref{alg:main-composite-pf}}\label{alg:main-composite-pf-gc-D}
\begin{algorithmic}[1]
 \Require Initial $\sigma_0=0,$ $x_0\in\mathbb{R}^n,$ and $L_0.$ Target accuracy $\varepsilon.$
 \State Run the subroutine $\mathcal{A}_s$ with $s=1,N_1=1,$ define $D_0^{p+\nu-1}\coloneqq \|\nabla f_1({x}_1^2)\|/L_{1,2}.$ 
  \For{$t=1,2\dots$} 
  \State Set $D_t=4 D_{t-1}.$
  \State Compute $\widehat{x}\leftarrow \text{\Cref{alg:main-composite-pf}} (f, x_0, {\varepsilon}/{[3(9D_t)^{p+\nu-1}]}).$
  \State If $\|\nabla f(\widehat{x})\|\leq \varepsilon,$ then {\bf terminate} with $\widehat{x}.$
     \EndFor
\end{algorithmic}
\end{algorithm}
 We have the following convergence guarantee for \Cref{alg:main-composite-pf-gc-D}.
\begin{theorem}
    Suppose the assumptions in \Cref{thm:p-order-gradient-optimal-h-pf}. Then \Cref{alg:main-composite-pf-gc-D} computes a solution $\widehat{x}$ such that $\|\nabla f(\widehat{x})\|\leq \varepsilon$
    within at most \eqref{eqn:perfect-D}
calls to the $p$-th-order oracle.
\end{theorem}
\begin{proof}
Observe that after at most $T$ calls to \Cref{alg:main-composite-pf}, \Cref{alg:main-composite-pf-gc-D} 
will terminate, where $T$ satisfies
$D_T=4^{T}D_0\geq \operatorname{dist}(x_0, X^*), $ and $D_{T-1}=4^{T-1}D_0\leq \operatorname{dist}(x_0, X^*),$ therefore, 
we have $T-1\leq \lceil  \log_4{\operatorname{dist}(x_0, X^*)}/{D_0}\rceil.$
By \Cref{thm:p-order-gradient-optimal-h-pf},
in total, \Cref{alg:main-composite-pf-gc-D} requires
   \begin{align*}
&c\textstyle\sum_{t=1}^{T}\left[\tfrac{c_{\mathcal{A}} \max\{pL_{p,\nu}(f), L_0, \theta \}(p+\nu)[3(9D_0\cdot 4^t)^{p+\nu-1}]}{4{\varepsilon}}\right]^{\frac{1}{p+\nu}}\notag\\
&\leq c\left[\tfrac{c_{\mathcal{A}} \max\{pL_{p,\nu}(f), L_0, \theta \}(p+\nu)[3(9D_0)^{p+\nu-1}]}{4{\varepsilon}}\right]^{\frac{1}{p+\nu}}\tfrac{4^{{T(p+\nu-1)}/{(p+\nu)}}\cdot4^{{(p+\nu-1)}/{(p+\nu)}}}{4^{{(p+\nu-1)}/{(p+\nu)}}-1}\notag\\
&\leq c\left[\tfrac{c_{\mathcal{A}} \max\{pL_{p,\nu}(f), L_0, \theta \}(p+\nu)[3(9\operatorname{dist}(x_0, X^*))^{p+\nu-1}]}{4{\varepsilon}}\right]^{\frac{1}{p+\nu}}\tfrac{16^{{(p+\nu-1)}/{(p+\nu)}}}{4^{{(p+\nu-1)}/{(p+\nu)}}-1}
\end{align*}
calls to the $p$-th-order oracle, where $c$ is a universal constant.
\end{proof}
\section{Gradient Minimization Framework for Uniformly Convex Problems}\label{sec:Global high-order non-degeneracy}

In this section, leveraging the (sub)gradient norm convergence results of the high-order AR algorithms established in \Cref{sec:Optimal high-order methods for structured convex problems} and \Cref{sec:parameter-free-convex} for convex composite problems, we obtain fast (sub)gradient-norm rates for composite problems that satisfy a suitable regularity condition.
\subsection{High-order AR for uniformly convex composite problems}
Consider a class of composite problems \eqref{eqn:composite-problem},
and in addition, $f$ is uniformly convex 
on $\mathbb{R}^n$ of degree $q$ with 
parameter $\sigma_{q}(f)>0$
\cite{polyak1966existence, aze1995uniformly,vladimirov1978uniformly,zǎlinescu1983uniformly}, i.e.,
\begin{align}\label{eqn:f-uniform-convexity}
    f(x)\geq f(y)+\langle g_y, x-y\rangle+\tfrac{\sigma_{q}(f)}{q}\|x-y\|^q,\quad \forall\, x, y\in\mathbb{R}^n, \quad \forall\,g_y\in\partial f(y),
\end{align}
where $q\geq 2.$ Here, we adopt the characterization of uniform convexity via the subgradient inequality; see for example \cite[Proposition~17.26]{bauschke2020correction}.
\begin{algorithm}[!ht]
\caption{High-order AR for uniformly convex functions }\label{alg:main-uniformconvex}
\begin{algorithmic}[1]
 \Require 
 Initial point $x_0\in\mathbb{R}^n$ and initial subgradient $\nu_0\in\partial f(x_0);$
 $L_{p+1}(g); \sigma_{q}(f).$
  \For{$k=1, 2\dots, $} 
    \State 
    Compute an approximate solution $(x_{k},\nu_k)$ of $f$ with initialization $(x_{k-1},\nu_{k-1})$
    by running high-order AR framework \Cref{alg:main-composite} for $m_k$ iterations, where
    \begin{align*}
        m_k=\max\left\{\left(\tfrac{qL_{p+1}(g)}{{\sigma_{q}(f)}}\right)^{\frac{2}{3p+1}}\left({\tfrac{q\|\nu_{k-1}\|}{2\sigma_{q}(f)}}\right)^{\frac{2(p-q+1)}{(3p+1)(q-1)}}, 1\right\}.
    \end{align*}
    \smallskip
    \State If $\|\nu_k\|\leq \varepsilon,$ then {\bf terminate} with $\widehat{x}=x_k.$
     \EndFor
 \State \Return{$\widehat{x}$}
\end{algorithmic}
\end{algorithm}
To solve this problem,
 we employ a restart scheme: restart the high-order AR framework \Cref{alg:main-composite}
whenever the (sub)gradient norm decreases by half. Specifically, we consider \Cref{alg:main-uniformconvex}, where 
for each epoch $k,$ we run $m_k$
iterations of the high-order AR framework \Cref{alg:main-composite} with the initialization 
$x_{k-1},$ i.e.,
the output of the previous epoch, $k-1.$
With a suitable selection of the epoch lengths \(m_k\), determined by the upper smoothness order \(p\) and the lower uniform convexity order \(q\), we obtain a sequence of subgradients with exponentially decaying norms across epochs.
We next present the convergence guarantee for \Cref{alg:main-uniformconvex}.
\begin{proposition}\label{eqn:prop:uniform-convexity}
 Suppose that \(\|\nu_0\|\ge \varepsilon\).
Then \Cref{alg:main-uniformconvex} computes a point \(\widehat{x}\) such that there exists
\(\widehat{\nu}\in\partial f(\widehat{x})\) with \(\|\widehat{\nu}\|\le \varepsilon\)
within at most
\begin{align*}
\mathcal{O}\left(\left[\tfrac{(p+1)L_{p+1}(g)}{{\sigma_{q}(f)}}\right]^{\frac{2}{3p+1}}\log \tfrac{\|\nu_0\|}{\varepsilon}\right)
\end{align*}
calls to the $p$-th-order oracle
if $q=p+1;$ 
and  at most
\begin{align*}
  \mathcal{O}\left(  \left[\tfrac{qL_{p+1}(g)}{{\sigma_{q}(f)}}\right]^{\frac{2}{3p+1}}\left(\tfrac{q}{\sigma_{q}(f)}\right)^{{\frac{2(q-1-p)}{(3p+1)(q-1)}}}\left(\tfrac{1}{\varepsilon}\right)^{{\frac{2(q-1-p)}{(3p+1)(q-1)}}}\right)
\end{align*}
calls to the $p$-th-order oracle
if $q>p+1;$
and at most
\begin{align*}
   \mathcal{O}\left( \tfrac{1}{\log{\left(\frac{p}{q-1}\right)}}{\log\log\left[\tfrac{\sigma_q(f)}{{q\varepsilon}}{\left(\tfrac{\sigma_q(f)}{qL_{p+1}(g)}\right)^{\frac{q-1}{p+1-q}}}\right]}+\left(\tfrac{qL_{p+1}(g)}{{\sigma_{q}(f)}}\right)^{\frac{2}{3p+1}}\left(\tfrac{q\|\nu_0\|}{\sigma_{q}(f)}\right)^{{\frac{2(p-q+1)}{(3p+1)(q-1)}}}\right)
\end{align*}
calls to the $p$-th-order oracle
if $q<p+1.$ 
\end{proposition}
\begin{proof}
By adding two copies of the inequality \eqref{eqn:f-uniform-convexity}, we have
\begin{align}\label{eqn:lower-uniform}
    \langle x-y, g_x-g_y\rangle\geq \tfrac{2\sigma_{q}(f)}{q}\|x-y\|^{q}, \quad \forall\,\, x, y\in\mathbb{R}^n,\quad\forall\,g_x\in\partial f(x), \,g_y\in\partial f(y).
\end{align}
Substituting $x=x_{k-1}$ and $y=x^*$ into \eqref{eqn:lower-uniform}, and 
using the Cauchy-Schwarz inequality and the optimality condition $0\in\partial f(x^*),$ we can choose $g_{x^*}=0$ and thus,
we have
\begin{align}\label{eqn:initial-optimality-bound-q}
  \|x_{k-1}-x^*\|\leq  \left({\tfrac{q\|g_{k-1}\|}{2\sigma_{q}(f)}}\right)^{\frac{1}{q-1}},\quad \forall\,g_{k-1}\in\partial f(x_{k-1}).
\end{align}
Let us prove by induction that there exists a sequence $\{\nu_k\}_{k\geq 1}$ such that
$\nu_k\in\partial f(x_k)$ and
$\|\nu_k\|\leq \|\nu_0\|/2^{k}.$
Suppose it holds for $k-1.$
At epoch $k,$ we restart \Cref{alg:main-composite} with the initial point $x_{k-1}.$ 
By \Cref{thm:p-order-gradient-optimal}, 
 \Cref{alg:main-composite} can compute an approximate solution 
 $x_k$ together with $\nu_k\in\partial f(x_k)$
  such that
\begin{align*}
     \|\nu_k\|&\,\,\,\leq m_k^{-\frac{3p+1}{2}}{L_{p+1}(g)\|x_{k-1}-x^*\|^{p}}{{}}\overset{\eqref{eqn:initial-optimality-bound-q}}{\leq}m_k^{-\frac{3p+1}{2}}{L_{p+1}(g)}{{}}\left({\tfrac{q\|g_{k-1}\|}{2\sigma_{q}(f)}}\right)^{ {\frac{p-q+1}{q-1}}}{\tfrac{q\|g_{k-1}\|}{2\sigma_{q}(f)}}.
\end{align*}
Given that $g_{k-1}$ can be any subgradient of $f$ at the point $x_{k-1},$ we choose  $g_{k-1}=\nu_{k-1}.$
By the choice of $m_k$ in \Cref{alg:main-uniformconvex}, we have $\|\nu_k\|\leq {{\|\nu_{k-1}\|}/{2}}\leq \|\nu_{0}\|/2^{k}.$
We next bound the iteration complexity.

{\noindent \it Case I}: $q=p+1.$
In this case, $m_k\equiv m_0=\left[{qL_{p+1}(g)}/{{\sigma_{p+1}(f)}}\right]^{\frac{2}{3p+1}},$ so the convergence rate is linear.
We can obtain an approximate solution $\widehat{x}$ such that $\|\widehat{\nu}\|\leq \varepsilon$
within at most $\mathcal{O}(\log {\|\nu_0\|}/{\varepsilon})$ epochs.

{\noindent \it Case II}: $q>p+1.$ In this case,  $m_k$ increases exponentially, and to reach a point ${x}_K$ such that $\|\nu_K
\|\leq \varepsilon,$
we need
\begin{align*}
\textstyle\sum_{k=1}^{K}m_k&=\left(\tfrac{qL_{p+1}(g)}{{\sigma_{q}(f)}}\right)^{\frac{2}{3p+1}}\left(\tfrac{\sigma_{q}(f)}{q\|\nu_0\|}\right)^{\frac{2(q-p-1)}{(3p+1)(q-1)}}\textstyle\sum_{k=1}^{K}2^{\frac{2(q-p-1)(k-1)}{(3p+1)(q-1)}}\notag\\
&\overset{\text{(a)}}{\leq }\left(\tfrac{qL_{p+1}(g)}{{\sigma_{q}(f)}}\right)^{\frac{2}{3p+1}}\left(\tfrac{2\sigma_{q}(f)}{q}\right)^{{\frac{2(q-p-1)}{(3p+1)(q-1)}}}\tfrac{1}{2^{\frac{2(q-p-1)}{(3p+1)(q-1)}}-1}\left(\tfrac{1}{\varepsilon}\right)^{{\frac{2(q-p-1)}{(3p+1)(q-1)}}}
\end{align*}
calls to the $p$-th-order oracle, where in (a), we used $\varepsilon<\|\nu_{K-1}\|\leq \|\nu_0\|/2^{K-1}.$ 

{\noindent \it Case III:} $q < p + 1.$ In this case, $m_k$ decays exponentially until it reaches $1,$ then, the algorithm converges superlinearly. Thus, the algorithm has two phases. Let $k = K_1$ denote the final iteration count of Phase~I, where $m_{K_1} = 1$ and $m_{K_1-1} > 1,$ we have
\begin{align}\label{eqn:m_K1=1}
m_{K_1}=\max\left\{\left(\tfrac{qL_{p+1}(g)}{{\sigma_{q}(f)}}\right)^{\frac{2}{3p+1}}\left(\tfrac{q\|\nu_{K_1-1}\|}{2\sigma_{q}(f)}\right)^{{\frac{2(p-q+1)}{(3p+1)(q-1)}}}, 1\right\}=1.
\end{align}
Then, at epoch $K_1,$ by \Cref{thm:p-order-gradient-optimal}, we have
\begin{align}\label{eqn:h-bound}
    \|\nu_{K_1}\|&\,\,\,\,\leq
    {L_{p+1}(g)\|x_{K_1-1}-x^*\|^{p}}\notag\\
    &\,\overset{\eqref{eqn:initial-optimality-bound-q}}{\leq}{L_{p+1}(g)\left({\tfrac{q\|\nu_{K_1-1}\|}{2\sigma_{q}(f)}}\right)^{\frac{p}{q-1}}}\overset{\eqref{eqn:m_K1=1}}{\leq}  \tfrac{1}{2}\left(\tfrac{1}{L_{p+1}(g)}\right)^{{\frac{q-1}{p+1-q}}}\left(\tfrac{{\sigma_{q}(f)}}{q}\right)^{{\frac{p}{p+1-q}}}.
\end{align}
The total number of iterations in Phase I is
\begin{align*}
\textstyle\sum_{k=1}^{K_1}m_k&\leq\textstyle\sum_{k=1}^{K_1-1}\left(\tfrac{qL_{p+1}(g)}{{\sigma_{q}(f)}}\right)^{\frac{2}{3p+1}}\left(\tfrac{q\|\nu_0\|}{2^{k-1}\sigma_{q}(f)}\right)^{{\frac{2(p-q+1)}{(3p+1)(q-1)}}}+1\notag\\
&\leq c\left(\tfrac{qL_{p+1}(g)}{{\sigma_{q}(f)}}\right)^{\frac{2}{3p+1}}\left(\tfrac{q\|\nu_0\|}{\sigma_{q}(f)}\right)^{{\frac{2(p-q+1)}{(3p+1)(q-1)}}}.
\end{align*}
Next, \Cref{alg:main-uniformconvex} enters  Phase II: the superlinear convergence phase. For all $k\geq K_1,$ $m_k=1,$ furthermore,
by \Cref{thm:p-order-gradient-optimal} and \eqref{eqn:initial-optimality-bound-q}, we have 
$$\|\nu_{k}\|\leq L_{p+1}(g)\left(\tfrac{q\|\nu_{k-1}\|}{2\sigma_{q}(f)}\right)^{\frac{p}{q-1}}.$$
Define $H\coloneqq{L_{p+1}(g)}^{\frac{q-1}{q-(p+1)}}\left(\tfrac{q}{2\sigma_q(f)}\right)^{\frac{p}{q-(p+1)}},$
thus, we have $$\tfrac{ \|\nu_k\|}{H}\leq \left(\tfrac{\|\nu_{k-1}\|}{H}\right)^{\frac{p}{q-1}}\quad\text{and}\quad \tfrac{\|\nu_{K_1}\|}{H}\overset{\eqref{eqn:h-bound}}{\leq} \tfrac{1}{2}. $$
Therefore, for all epochs $k\geq K_1,$
we have $ \|\nu_k\|\leq H2^{-(\frac{p}{q-1})^{k-K_1}}$ and $m_k=1.$  
For any $\varepsilon>0,$ to reach a point $\widehat{x}$ that satisfies $\|\widehat{\nu}\|\leq\varepsilon,$ we need 
 $$t\geq \tfrac{1}{\log{\left(\frac{p}{q-1}\right)}}\left({\log\log\tfrac{H}{\varepsilon}- \log \log2}\right){}+\textstyle\sum_{k=1}^{K_1}m_k$$ iterations. This concludes the proof.
\end{proof}
Depending on the relation between $p+1$ and $q,$ \Cref{alg:main-uniformconvex} can exhibit linear, sublinear or superlinear convergence rates.
Such complexities in terms of the gradient norm have not been derived before.   
\subsection{Parameter-free gradient minimization for uniformly convex functions}
In this section, we develop a parameter-free, inexact method for uniformly convex problems that does not assume knowledge of the Lipschitz constant or the uniform convexity parameter. We build on \Cref{sec:parameter-free-convex} and use a guess-and-check scheme to estimate the uniform convexity parameter.

For simplicity, we will focus on the unconstrained problem \eqref{eqn:main} where $f$ is $p$-times differentiable with a Lipschitz continuous $p$-th derivative \eqref{eqn:smoothness-p}, and assume that $f$  is uniformly convex with order $q=p+1.$ Below we present Algorithm~\ref{alg:main-uniformconvex-pf}, a parameter-free and inexact algorithm tailored to uniformly convex functions, together with its convergence guarantees.
\begin{algorithm}[!ht]
\caption{{Parameter-free uniformly convex accumulative regularization} }\label{alg:main-uniformconvex-pf}
\begin{algorithmic}[1]
 \Require 
 Initial point $x_0;$ initial uniform convexity guess $  \sigma_{p+1, 0};$ initial line-search value $ L_0;$ target accuracy $\varepsilon.$
  \For{$t=1, 2\dots, $} 
    \State 
    Set $(x_t, L_t)\leftarrow\text{\Cref{alg:main-composite-pf}}(f, x_{t-1}, {\sigma_{p+1,t-1}}/{[3(p+1)9^p]}, L_{t-1}).$
   \State If $\|\nabla f(x_t)\|>{\|\nabla f(x_{t-1})\|}/{2},$ then $\sigma_{p+1, t}={\sigma_{p+1, t-1}}/{4}.$
    \smallskip
    \State If $\|\nabla f(x_t)\|\leq \varepsilon,$ then {\bf terminate} with $\widehat{x}=x_t.$
     \EndFor
 \State \Return{$\widehat{x}.$}
\end{algorithmic}
\end{algorithm}

\begin{theorem}\label{thm:mu-pf}
   Suppose the assumptions in \Cref{thm:p-order-gradient-optimal-h-pf} hold and that $f$ is uniformly convex of order $p+1$ with parameter $\sigma_{p+1}(f)$.
 Then \Cref{alg:main-uniformconvex-pf} computes a solution $\widehat{x}$ such that $\|\nabla f(\widehat{x})\|\leq \varepsilon$
    within at most 
    \begin{align*}
       \mathcal{O}\left(\left(\tfrac{\max\{{pL_{p+1}}(f), L_0, \theta \}}{\sigma_{p+1}(f)}\right)^{\frac{1}{p+1}}\left\{\left\lceil \log_4\tfrac{\sigma_{p+1,0}}{\sigma_{p+1}(f)}\right\rceil+\left\lceil \log_2\tfrac{\|\nabla f(x_0)\|}{\varepsilon}\right\rceil\right\}\right)
    \end{align*}
calls to the $p$-th-order oracle, where $\theta$ is a user-defined inexactness parameter in \Cref{as:stronger-ANPE-tensor-pf}.
\end{theorem}
\begin{proof}
    {\bf i)}  By 
\Cref{thm:p-order-gradient-optimal-h-pf} with $\nu=1$ and uniform convexity with $q=p+1,$ for each epoch $t,$
    we have \begin{align}\label{eqn:intermedia-uniformly-convex}
        \|\nabla f(x_t)\|&\,\,\leq \tfrac{\sigma_{p+1, t-1}}{p+1}\|x_{t-1}-x^*\|^{p}
\overset{\eqref{eqn:initial-optimality-bound-q}}{\leq}\tfrac{\sigma_{p+1, t-1}\|\nabla f(x_{t-1})\|}{2\sigma_{p+1}(f)}
    \end{align}
    after $\mathcal{O}([{\max\{{pL_{p+1}}(f), L_0, \theta \}}/{\sigma_{p+1, t-1}}]^{\frac{1}{p+1}})$
{calls to the $p$-th-order oracle.} 

{\bf ii)} At epoch $t,$ if $\|\nabla f(x_{t})\|>{\|\nabla f(x_{t-1})\|}/{2},$ 
 then by \eqref{eqn:intermedia-uniformly-convex}, we conclude the current $\sigma_{p+1, t-1}$ overestimates the true uniform convexity parameter $\sigma_{p+1}(f).$
Hence, we set
$\sigma_{p+1,t}={{\sigma_{p+1,t-1}}}/{{4}}.$
However,  the estimate $\sigma_{p+1,t}$ will no longer change after at most $\lceil \log_4\tfrac{\sigma_{p+1,0}}{\sigma_{p+1}(f)}\rceil$ epochs. Denote the set of such epoch indices by $\mathcal{T}_1.$ On the other hand, at epoch $t,$ if $\|\nabla f(x_t)\|\leq {\|\nabla f(x_{t-1})\|}/{2},$
 then by \eqref{eqn:intermedia-uniformly-convex}, we conclude the current $\sigma_{p+1, t-1}$ is small enough, and thus in the next epoch $t,$ we have $\sigma_{p+1, t}=\sigma_{p+1, t-1}.$  Furthermore, 
 $\|\nabla f(x_t)\|\leq\varepsilon$
 after at most $\lceil \log_2\tfrac{\|\nabla f(x_0)\|}{\varepsilon}\rceil$ epochs. Denote the set of such epoch indices by $\mathcal{T}_2.$  
Therefore, \Cref{alg:main-uniformconvex-pf} terminates in at most $$\lceil \log_4\tfrac{\sigma_{p+1,0}}{\sigma_{p+1}(f)}\rceil+\lceil \log_2\tfrac{\|\nabla f(x_0)\|}{\varepsilon}\rceil$$ epochs of the parameter-free AR method \Cref{alg:main-composite-pf}. Denote $\mathcal{T}\coloneqq \mathcal{T}_1\cup \mathcal{T}_2.$ Combining {\bf i), ii)}, we conclude the total number of $p$-th-order oracle calls is
 \begin{align*}
     &\textstyle\sum_{t\in \mathcal{T}}c_p\left[\tfrac{\max\{{pL_{p+1}}(f), L_0, \theta \}}{\sigma_{p+1, t-1}}\right]^{\frac{1}{p+1}}=c_p\max\{{pL_{p+1}}(f), L_0, \theta \}^{\frac{1}{p+1}}\textstyle\sum_{t\in \mathcal{T}}\left(\tfrac{1}{\sigma_{p+1, t-1}}\right)^{\frac{1}{p+1}}\notag\\
     &\overset{\text{(i)}}{\leq} c_p\left(\tfrac{\max\{{pL_{p+1}}(f), L_0, \theta \}}{\sigma_{p+1}(f)}\right)^{\frac{1}{p+1}} |\mathcal{T}_1|+c_p\left(\tfrac{\max\{{pL_{p+1}}(f), L_0, \theta \}}{\sigma_{p+1}(f)/4}\right)^{\frac{1}{p+1}} |\mathcal{T}_2|,
      % &\leq c_p\left(\tfrac{\max\{{pL_{p+1}}(f), L_0, \theta \}}{\sigma_{p+1, 0}}\right)^{\frac{1}{p+1}}\tfrac{4^{\frac{T_1}{p+1}}-1}{4^{\frac{1}{p+1}}-1}+\tfrac{c_p\max\{{pL_{p+1}}(f), L_0, \theta \}^{\frac{1}{p+1}}}{\sigma_{p+1}(f)^{\frac{1}{p+1}}}\lceil \log_2\tfrac{\|\nabla f(x_0)\|}{\varepsilon}\rceil\notag\\
      % &\overset{\text{(ii)}}{\leq}c_p\left(\tfrac{\max\{{pL_{p+1}}(f), L_0, \theta \}}{\sigma_{p+1}(f)}\right)^{\frac{1}{p+1}}\tfrac{4^{\frac{|\mathcal{T}|}{p+1}}}{4^{\frac{1}{p+1}}-1}+c_p\left(\tfrac{\max\{{pL_{p+1}}(f), L_0, \theta \}}{\sigma_{p+1}(f)}\right)^{\frac{1}{p+1}}\lceil \log_2\tfrac{\|\nabla f(x_0)\|}{\varepsilon}\rceil,
 \end{align*}
 where $c_p=3(p+1)9^p,$
and in (i), we used the fact $\sigma_{p+1, t}>\sigma_{p+1}(f),$ for all $t\in \mathcal{T}_1,$ and 
$\sigma_{p+1, t-1}=\sigma_{p+1, t}$ for all $t\in\mathcal{T}_2.$
Suppose $t_0$ is the smallest index such that $\sigma_{p+1, t_0}=\sigma_{p+1, t},$ then by \eqref{eqn:intermedia-uniformly-convex}
we conclude that $\sigma_{p+1}(f)\leq\sigma_{p+1, t_0-1},$ and $\sigma_{p+1, t_0-1}=4\sigma_{p+1, t_0}.$ Therefore, for all $t\in \mathcal{T}_2,$ we have $\sigma_{p+1, t}\geq \sigma_{p+1}(f)/4.$
% and $\sigma_{p+1, T_1}=\sigma_{p+1, T_1-1}/4=\sigma_{p+1, 0}/{4^{T_1}}\geq\sigma_{p+1}(f)/4;$
Substituting the definitions of $\mathcal{T}_1$ and $\mathcal{T}_2$ concludes the proof.
\end{proof}
Observe that the rate matches the case where $L_{p+1}(f), \sigma_{p+1}(f)$ are known. Similar guess-and-check procedures for unknown $\sigma_{p+1}(f)$
can be extended to the case $p+1\neq q,$
and are thus omitted for brevity.
\section{Concluding Remarks}
In this paper, we establish a unified high-order accumulative regularization framework  for gradient-norm minimization on convex and uniformly convex objectives. By carefully designing regularization strategies and leveraging high-order information, the proposed algorithms bridge the long-standing gap between function-residual and gradient norm complexity guarantees. Our results demonstrate that the high-order framework effectively accelerates the existing slow convergence of the gradient norm by exploiting the fast function-residual convergences of the subroutine. Furthermore, we develop several parameter-free variants that achieve the same convergence rates without requiring prior knowledge of problem parameters. Overall, our framework attains the best-known gradient norm rates across a broad range of smoothness and curvature regimes, recovering classical first-order results as special cases while extending to previously unexplored high-order settings.

\begin{appendix}

\section{Proof of Lemma \ref{lem:prop-cubic}}
\label{sec:Gradient Complexity of the  Subroutine}
\begin{proof}
We aim to maintain relations $ \mathcal{R}_k^1,  \mathcal{R}_k^2$ for all $k\geq 1.$ 
We prove by induction.
When $k=1,$
by the definition of $f_1(x)$ in \Cref{alg:inner-loop},  we have $f_1^*=f(x_1)+\sqrt{1/{(L_3(f)+M)}}\left\|\nabla f(x_1)\right\| ^{\frac{3}{2}}.$
By choosing $A_1=1,$
we conclude $\mathcal{R}_1^1$ holds. Furthermore,  we have
\begin{align*}    f_1(x)&\,=
    f(x_1)+\tfrac{1}{\sqrt{L_3(f)+M}}\left\|\nabla f(x_1)\right\| ^{\frac{3}{2}}+\tfrac{C\|x-x_0\|^3}{6}\notag\\
    &\overset{\text{(a)}}{\leq} \min\limits_{y\in\mathbb{R}^n}\left[ f(y)+\tfrac{2L_3(f)\|y-x_0\|^3}{6}\right]+\tfrac{\left\|\nabla f(x_1)\right\| ^{\frac{3}{2}}}{\sqrt{L_3(f)+M}}+\tfrac{C\|x-x_0\|^3}{6}\notag\\
     &\,\leq f(x)+\tfrac{(2L_3(f)+C)\|x-x_0\|^3}{6}+\tfrac{\left\|\nabla f(x_1)\right\| ^{\frac{3}{2}}}{\sqrt{L_3(f)+M}},
\end{align*}
where in (a), we used \cite[Lemma 4]{nesterov2006cubic}, \cite[Lemma 6]{nesterov2008accelerating}.
 Thus,  $ \mathcal{R}_1^2$ holds. Assume that $\mathcal{R}_{k}^1$ and $\mathcal{R}_{k}^2$   hold for some $k\geq 1.$ For $k+1,$ we have
 \begin{align*}
      f_{k+1}(x)\,&\,=f_k(x)+a_k\left[ f(x_{k+1})+\langle \nabla f(x_{k+1}), x-x_{k+1}\rangle\right]\notag\\
      &\overset{\mathcal{R}_{k}^2}{\leq} A_k f(x)+\tfrac{(2L_3(f)+C)\|x-x_0\|^3}{6}+\tfrac{\left\|\nabla f(x_1)\right\| ^{\frac{3}{2}}}{\sqrt{L_3(f)+M}}\notag\\
      &\quad+a_k[f(x_{k+1})+\langle \nabla f(x_{k+1}), x-x_{k+1}\rangle]\notag\\
      &\overset{\text{(b)}}{\leq} A_{k+1} f(x)+\tfrac{(2L_3(f)+C)\|x-x_0\|^3}{6}+\tfrac{\left\|\nabla f(x_1)\right\| ^{\frac{3}{2}}}{\sqrt{L_3(f)+M}},
 \end{align*}
 where in (b), we used the convexity of $f$ and $A_{k+1}=A_k+a_k.$ Therefore $ \mathcal{R}_{k+1}^2$ holds. It remains to show $ \mathcal{R}_{k+1}^1$ holds, which
it is a straightforward modification of \cite{nesterov2008accelerating}, we include it here for completeness. 
By the definition of $f_{k+1}(x),$ we have
 \begin{align}\label{eqn:lower-psi-optimal-k+1}
f_{k+1}^*&\,\,{=}\min\limits_{x\in\mathbb{R}^n}\left\{ f_k(x)+a_k \left[ f(x_{k+1})+\langle \nabla f(x_{k+1}), x-x_{k+1}\rangle\right]\right\}.
 \end{align}
 We proceed with providing a lower bound on $f_k(x).$
By the definition of $f_k(x)$ in \Cref{alg:inner-loop}, we have $f_k(x)=\ell(x)+\textstyle\sum_{i=1}^{k-1} a_i\left( f(x_{i+1})+\langle \nabla f(x_{i+1}), x-x_{i+1}\rangle\right)+{C}\|x-x_0\|^3/6,$
therefore, $f_k(x)$ is a uniform convex function of degree 3 (modulus $C/12$). Thus, we have
\begin{align}\label{eqn:k-uniform-convexity}
f_k(x)\overset{\text{(c)}}{\geq}&   f_k^*+\tfrac{C}{12}\|x-\nu_k\|^3\
\overset{\mathcal{R}_k^1}{\geq}A_k f(x_k)+\textstyle\sum_{j=1}^k{A_{j}\tfrac{\left\|\nabla f(x_{j})\right\| ^{\frac{3}{2}}}{\sqrt{L_3(f)+M}}}+\tfrac{C\|x-\nu_k\|^3}{12},
\end{align}
where in (c), we used the uniform strong convexity of $f_k.$  Note that $\nu_k$ minimizes $f_k(x)$ defined in \Cref{alg:inner-loop}, and $\nabla f_k(\nu_k)=0.$
Substituting the lower bound for $f_k(x)$ into \eqref{eqn:lower-psi-optimal-k+1},
we have
\begin{align}\label{eqn:induction-R-1}
f_{k+1}^*
% &\overset{\eqref{eqn:psi_k_def}}{=}\min\limits_{x\in\mathbb{R}^n}\left\{ f_k(x)+a_k \left[ f(x_{k+1})+\langle \nabla f(x_{k+1}), x-x_{k+1}\rangle\right]\right\}\notag\\
&\overset{\eqref{eqn:k-uniform-convexity}}{\geq}\min\limits_{x\in\mathbb{R}^n}\left\{ A_k f(x_k)+{\textstyle\sum_{j=1}^kA_{j}\tfrac{\left\|\nabla f(x_{j})\right\| ^{\frac{3}{2}}}{\sqrt{L_3(f)+M}}}+\tfrac{C\|x-\nu_k\|^3}{12}\right.\notag\\
&\quad\left.+a_k\left[ f(x_{k+1})+\langle \nabla f(x_{k+1}), x-x_{k+1}\rangle\right]\right\}\notag\\
% &\overset{\text{(a)}}{\geq}\min\limits_{x\in\mathbb{R}^n}\{ A_{k+1} f(x_{k+1})+A_k\langle\nabla f(x_{k+1}), x_k-x_{k+1}\rangle\notag\\
% &\quad+\textstyle\sum_{j=1}^k{A_{j}\sqrt{\frac{1}{L+M}}\left\|\nabla f(x_{j})\right\| ^{\frac{3}{2}}}+\frac{C}{12}\|x-\nu_k\|^3+a_k\left[ \langle \nabla f(x_{k+1}), x-x_{k+1}\rangle\right]\notag\\
% &{=}\min\limits_{x\in\mathbb{R}^n}\left\{A_{k+1}f(x_{k+1})+\langle\nabla f(x_{k+1}), A_{k+1}y_k-a_k \nu_k-A_kx_{k+1}\rangle\right.\notag\\
% &\quad\left.+\textstyle\sum_{j=1}^k{A_{j}\sqrt{\tfrac{1}{L+M}}\left\|\nabla f(x_{j})\right\| ^{\frac{3}{2}}}+a_k\left[ \langle \nabla f(x_{k+1}), x-x_{k+1}\rangle\right]+\tfrac{C}{12}\|x-\nu_k\|^3\right\}\notag\\
&\overset{\text{(d)}}{\geq}A_{k+1}[f(x_{k+1})+\langle\nabla f(x_{k+1}), y_k-x_{k+1}\rangle]+\textstyle\sum_{j=1}^k{\tfrac{A_{j}\left\|\nabla f(x_{j})\right\| ^{\frac{3}{2}}}{\sqrt{L_3(f)+M}}}\notag\\
&\quad+\min\limits_{x\in\mathbb{R}^n}\left\{a_k\langle \nabla f(x_{k+1}), x-\nu_k\rangle+\tfrac{C}{12}\|x-\nu_k\|^3\right\},
\end{align}
where in (d), we used the convexity of $f$ and $A_{k+1}=A_k+a_k.$ 
By \cite[Lemma 4]{nesterov2006cubic}, \cite[Lemma 6]{nesterov2008accelerating},  if $M\ge 2L_3(f),$
we have
\begin{align}\label{eqn:lemma2-substute}
 \langle \nabla f(x_{k+1}), y_k-x_{k+1}\rangle\geq \sqrt{\tfrac{2}{L_3(f)+M}}\|\nabla f(x_{k+1})\|^{\frac{3}{2}}.
\end{align}
Furthermore, given that $a_k, C, M$ are chosen as in \Cref{lem:prop-cubic}, we have
\begin{align*}
    A_k&=A_{k-1}+a_k
    % =\textstyle\sum_{i=1}^k a_k
    =\tfrac{k(k+1)(k+2)}{6},\quad 
   a_k^{-\frac{3}{2}}A_{k+1}\geq \tfrac{2}{3},\quad \tfrac{1}{L_3(f)+M}\geq \tfrac{4}{C(\sqrt{2}-1)^2}.
\end{align*}
Therefore, we have
\begin{align}\label{eqn:intermedis-parameters-relation2}
&\tfrac{\left(\sqrt{2}-1\right)A_{k+1}\left\|\nabla f(x_{k+1})\right\| ^{\frac{3}{2}}}{\sqrt{L_3(f)+M}}{\geq}\tfrac{4a_k^{\frac{3}{2}}\left\|\nabla f(x_{k+1})\right\| ^{\frac{3}{2}}}{3\sqrt{C}}\notag\\
&\geq -\min\limits_{x\in\mathbb{R}^n}\left\{a_k \langle \nabla f(x_{k+1}), x-\nu_k\rangle+\tfrac{C\|x-\nu_k\|^3}{12}\right\}.
\end{align}
% Hence, if we choose $a_k, C$ such that
% \begin{align}\label{eqn:cubic-equation}
% &\left(\sqrt{2}-1\right)A_{k+1}\sqrt{\frac{1}{L+M}}\left\|\nabla f(x_{k+1})\right\| ^{\frac{3}{2}}\geq -\min\limits_{x\in\mathbb{R}^n}\left\{a_k \langle \nabla f(x_{k+1}), x-\nu_k\rangle+\frac{C}{12}\|x-\nu_k\|^3\right\},
% \end{align}
Substituting \eqref{eqn:lemma2-substute} and \eqref{eqn:intermedis-parameters-relation2} into \eqref{eqn:induction-R-1}, $\mathcal{R}_{k+1}^1$ holds.
% i.e., $$f_{k+1}^*{\geq}   A_{k+1}f(x_{k+1})+\textstyle\sum_{j=1}^kA_{j+1}\sqrt{\frac{1}{L+M}}\left\|\nabla f(x_{j+1})\right\| ^{\frac{3}{2}}.$$ 
Using $\mathcal{R}_{k}^1$ and $\mathcal{R}_{k}^2,$ we have
\begin{align*}
   & A_k f(x_k)+\textstyle\sum_{j=1}^kA_j\tfrac{\|\nabla f(x_j)\| ^{\frac{3}{2}}}{\sqrt{L_3(f)+M}}\overset{\text{$\mathcal{R}_{k}^1$}}{\leq} f_k^*\leq f_k(x)\notag\\
   &
   \overset{\text{$\mathcal{R}_{k}^2$}} {\leq}A_k f(x)+\left[\left(\tfrac{1}{2\sqrt{2}}+\tfrac{1}{6}\right)(L_3(f)+M)+\tfrac{C}{6}\right]\|x_0-x\|^3+\tfrac{\left\|\nabla f(x_1)\right\| ^{\frac{3}{2}}}{\sqrt{L_3(f)+M}}.
\end{align*}
Substituting the choice of $A_k, C, M,$ and $x=x^*$ concludes the proof.
\end{proof}   
For a general $p$-th-order accelerated tensor method, we can utilize \citep[Corollary~1]{nesterov2021implementable} and analyze a modified accelerated tensor method to derive a slow gradient complexity $\mathcal{O}(1/\varepsilon^{p})$, which satisfies \Cref{as:stronger-ANPE-tensor} and thus can be used as a subroutine in the high-order AR framework \Cref{alg:main-composite}. The proof follows a similar idea to \cite{nesterov2021implementable}, modifying the estimating sequence; it is also analogous to the accelerated CNM in how the gradient norm output is obtained, so we omit the details.

\end{appendix}
% \section*{Acknowledgments}
% \Yao{here}

\bibliographystyle{siamplain}
\bibliography{references}

\end{document}